\documentclass[11pt]{article}

\usepackage[margin=1.1in]{geometry}
\usepackage{amsmath,amsfonts,amssymb}
\usepackage{amsthm}
\IfFileExists{lmodern.sty}{\usepackage{lmodern}\usepackage[T1]{fontenc}}{}
\usepackage[expansion=false]{microtype}
\usepackage{graphicx}
\usepackage{epstopdf}
\usepackage{booktabs}   
\usepackage{tabularx}   
\usepackage{algorithm,algpseudocode}

\usepackage{enumitem}
\usepackage{xcolor}
\usepackage[colorlinks=true, linkcolor=blue, citecolor=blue, urlcolor=blue]{hyperref}
\usepackage{cleveref}

\numberwithin{equation}{section}
\theoremstyle{plain}
\newtheorem{theorem}{Theorem}[section]
\newtheorem{lemma}[theorem]{Lemma}
\newtheorem{corollary}[theorem]{Corollary}
\newtheorem{proposition}[theorem]{Proposition}

\theoremstyle{definition}
\newtheorem{remark}[theorem]{Remark}

\newtheorem{example}[theorem]{Example}
\crefname{hypothesis}{Hypothesis}{Hypotheses}

\newenvironment{keywords}{\par\medskip\noindent\textbf{Keywords.}\ }{\par}

\DeclareGraphicsExtensions{.eps,.pdf,.png,.jpg}

\hypersetup{
	pdftitle={Sparse Plus Low-Rank Matrix Embedding with Applications in Cancer Radiotherapy Optimization},
	pdfauthor={Mojtaba Tefagh, Gourav Jhanwar, and Masoud Zarepisheh}
}

\title{Sparse plus low-rank matrix embedding with applications in cancer radiotherapy optimization\thanks{This work was supported in part by NIH Cancer Center Support Grant/Core Grant (P30 CA008748).}}

\author{Mojtaba Tefagh\thanks{School of Informatics, the University of Edinburgh, Edinburgh, Scotland.}
	\and Gourav Jhanwar\thanks{Department of Medical Physics, Memorial Sloan Kettering Cancer Center, New York, NY, USA.}
	\and Masoud Zarepisheh\thanks{Department of Medical Physics, Memorial Sloan Kettering Cancer Center, New York, NY, USA. Correspondence e-mail: zarepism@mskcc.org.}}
\date{}

\makeatletter
	\newcounter{crefprobe}
	\def\cref@probe@first#1#2\@nil{\def\cref@probe@char{#1}}
	\def\cref@lb{[}
	\newif\ifcref@needsrepair
	\AtBeginDocument{%
		\begingroup
		\def\cref@currentlabel{}%
		\refstepcounter{crefprobe}%
		\expandafter\cref@probe@first\cref@currentlabel\cref@nomatch\@nil
		\ifx\cref@probe@char\cref@lb
		\global\cref@needsrepairfalse
		\else
		\global\cref@needsrepairtrue
		\fi
		\endgroup
		\ifcref@needsrepair
		\def\refstepcounter@noarg#1{%
			\cref@old@refstepcounter{#1}%
			\@ifundefined{cref@#1@alias}%
			{\def\@tempa{#1}}%
			{\def\@tempa{\csname cref@#1@alias\endcsname}}%
			\cref@constructprefix{#1}{\cref@result}%
			\protected@edef\cref@currentlabel{%
				[\@tempa][\arabic{#1}][\cref@result]%
				\csname p@#1\endcsname\csname the#1\endcsname}}%
		\def\refstepcounter@optarg[#1]#2{%
			\cref@old@refstepcounter{#2}%
			\@ifundefined{cref@#1@alias}%
			{\def\@tempa{#1}}%
			{\def\@tempa{\csname cref@#1@alias\endcsname}}%
			\cref@constructprefix{#2}{\cref@result}%
			\protected@edef\cref@currentlabel{%
				[\@tempa][\arabic{#2}][\cref@result]%
				\csname p@#2\endcsname\csname the#2\endcsname}}%
		\fi
	}
	\makeatother
	
\begin{document}

		\maketitle
		
		\begin{abstract}
			
			Decomposing a matrix into sparse and low-rank components is central to robust principal component analysis and has broad applications in machine learning, signal processing, and computer vision. Classical formulations seek to recover the underlying sparse and low-rank structure. We instead introduce \emph{sparse-plus-low-rank matrix embedding} (SLME), whose goal is to construct a computationally efficient surrogate for a large dense matrix, without requiring its components to be interpretable. Given $A\in\mathbb{R}^{m\times n}$, SLME approximates $A \approx S+HW$, where $S$ is sparse, $H\in\mathbb{R}^{m\times r}$, $W\in\mathbb{R}^{r\times n}$, and $r\ll \min\{m,n\}$. The resulting matrix-vector product can be evaluated as $Sx+H(Wx)$ in $\mathrm{nnz}(S)+r(m+n)$ operations, rather than the $mn$ operations required by $Ax$. Our primary motivation arises from optimization problems in cancer radiotherapy treatment planning, where a large dense \emph{dose-influence matrix} is a major computational bottleneck.
			
			We formulate SLME as a bi-objective nonconvex optimization problem that balances approximation error against the computational cost of the downstream tasks. We then develop \emph{R3-Trust}, an efficient trust-region algorithm that approximates the Pareto frontier in a single parameter-free run. Each point on the resulting frontier provides a sparse-plus-low-rank representation with a different balance between accuracy and downstream computational cost. Experiments on clinical radiotherapy matrices show that decompositions obtained from existing recovery-oriented formulations can be suboptimal for the embedding objective. Conversely, experiments on synthetic instances demonstrate that SLME and R3-Trust can also be applied to sparse-plus-low-rank recovery, where they compare favorably with state-of-the-art recovery methods in both reconstruction accuracy and computational time. All code and data are publicly available at \url{https://github.com/Radiotherapy-Optimization/SLME-matrix-embedding}.
			
		\end{abstract}
		
		\begin{keywords}
			matrix decomposition, robust principal component analysis, principal component pursuit, radiotherapy optimization, low-rank plus sparse, sparse plus low-rank
		\end{keywords}
		
		\setlength{\parindent}{1em}
		\setlength{\parskip}{0.5em}
		
		\section{Introduction}\label{sec:introduction}
		
		Sparse-plus-low-rank matrix decompositions arise throughout machine learning,
		signal processing, computer vision, and large-scale data analysis. Existing
		formulations are primarily recovery-oriented: they seek to identify the underlying latent
		low-rank and sparse components from observations corrupted by noise or
		outliers~\cite{Candes2009-fm,Chandrasekaran2011-vc,Zhou2014-ug}. In this work,
		we consider a fundamentally different objective. Rather than recovering
		interpretable components, we seek a computationally efficient representation
		of a given large dense matrix.
		
		Our motivating application is \emph{cancer radiotherapy treatment planning optimization}~\cite{shepard1999optimizing,romeijn2005column}, where many
		optimization problems can be expressed in the following form:
		
		\begin{equation}\label{RO-Problem}
			\min_x \ f(Ax,x)
			\quad \text{s.t.} \quad
			g(Ax,x)\leq 0,\qquad x\in X,
		\end{equation}
		where $x$ represents the treatment-machine parameters,
		$A\in\mathbb{R}^{m\times n}$ is a patient-specific \emph{dose-influence
			matrix}, and $Ax$ is the radiation-dose distribution delivered to the patient.
		The functions $f$ and $g$ represent the treatment-planning objectives and
		clinical constraints, respectively. The matrix $A$ is typically large and
		dense, with $m$ ranging from approximately $10^5$ to $5\times 10^5$ and $n$
		ranging from approximately $10^3$ to $10^5$. Consequently, repeated
		multiplications by $A$ and $A^\top$ often constitute the primary computational
		bottleneck in treatment-planning optimization algorithms.
		
		A common strategy in clinical implementations is to sparsify $A$ by discarding
		entries whose magnitudes fall below a prescribed threshold. Although this
		entrywise truncation reduces storage and the cost of computations involving $A$, it may remove a
		large number of individually small entries whose cumulative contribution to the
		radiation-dose distribution is substantial. In recent
		work~\cite{tefagh2025compressed}, we showed empirically that this loss can be
		reduced by approximating $A$ as $A\approx S+HW$, where $S$ is sparse and $HW$
		is low rank. In numerical experiments, this representation
		reduced solution times while producing higher-quality radiotherapy plans than entrywise sparsification alone.
		
		The decomposition used in that study, however, was constructed heuristically.
		Entries exceeding a prescribed magnitude threshold were assigned to $S$, after
		which a truncated singular value decomposition of a specified rank was applied
		to the residual to construct $HW$. As a result, the sparsity, rank,
		representation cost, and approximation error were controlled only indirectly
		through separately selected parameters. The present work develops a principled
		framework that directly balances these competing quantities.
		
		Motivated by this application, we introduce \emph{sparse-plus-low-rank matrix
			embedding} (SLME). Given a large dense matrix $A\in\mathbb{R}^{m\times n}$,
		SLME constructs an approximation
		
		\begin{equation}\label{SLME-representation}
			A\approx S+HW,
		\end{equation}
		where $S\in\mathbb{R}^{m\times n}$ is sparse,
		$H\in\mathbb{R}^{m\times r}$, $W\in\mathbb{R}^{r\times n}$, and
		$r\ll\min\{m,n\}$. The resulting matrix-vector product can be evaluated as
		$Sx+H(Wx)$ at a cost proportional to
		$\operatorname{nnz}(S)+r(m+n)$, rather than the $mn$ operations required for
		a dense multiplication by $A$. Thus, SLME embeds a large dense matrix into a substantially more compact structured representation.
		
		\Cref{fig:motivating_example} illustrates the potential of SLME using a
		dose-influence matrix $A$ arising in the treatment planning of a patient with
		lung cancer. The horizontal axis reports the representation size relative to
		$\operatorname{nnz}(A)$, while the vertical axis shows the relative
		Frobenius-norm error,
		$100\times \|A-(S+HW)\|_{\text{F}}/\|A\|_{\text{F}}$. We compare sparse-only, low-rank-only, and
		sparse-plus-low-rank approximations. For the sparse-only approximation, $A$ is
		thresholded entrywise, and the curve is obtained by varying the threshold. For
		the low-rank-only approximation, truncated singular value decompositions of
		varying rank are used, with representation size $r(m+n)$. For SLME, the factors
		$(S,H,W)$ are computed using the R3-Trust algorithm introduced later, with
		representation size $\operatorname{nnz}(S)+r(m+n)$.
		
		The representative point marked by a star in the left panel requires only about
		$1.5\%$ as many stored scalar values as $A$, that is,
		$\operatorname{nnz}(S)+r(m+n)\approx
		0.015\,\operatorname{nnz}(A)$, while achieving a relative approximation error
		of approximately $2.5\%$. Across all representation budgets shown, SLME also
		achieves substantially smaller error than either the sparse-only approximation
		in the left panel or the low-rank-only approximation in the right panel. These
		results suggest that the sparse component efficiently captures the large,
		localized entries of $A$, while the low-rank component captures its distributed
		global structure; neither component alone provides comparable approximation
		accuracy at the same representation cost.
		
		\begin{figure}[htb!]
			\centering
			\includegraphics[width=0.45\linewidth]
			{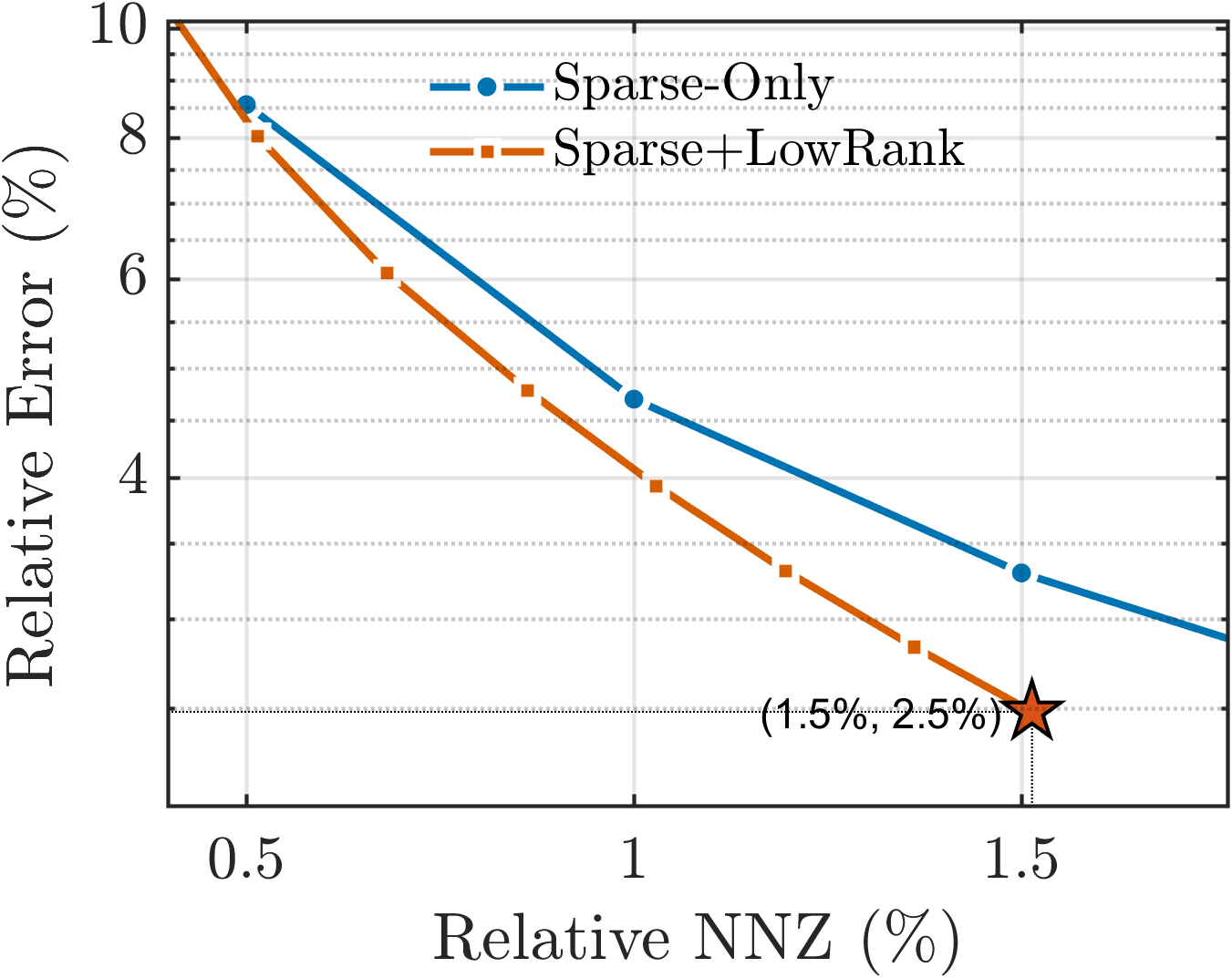}
			\includegraphics[width=0.45\linewidth]
			{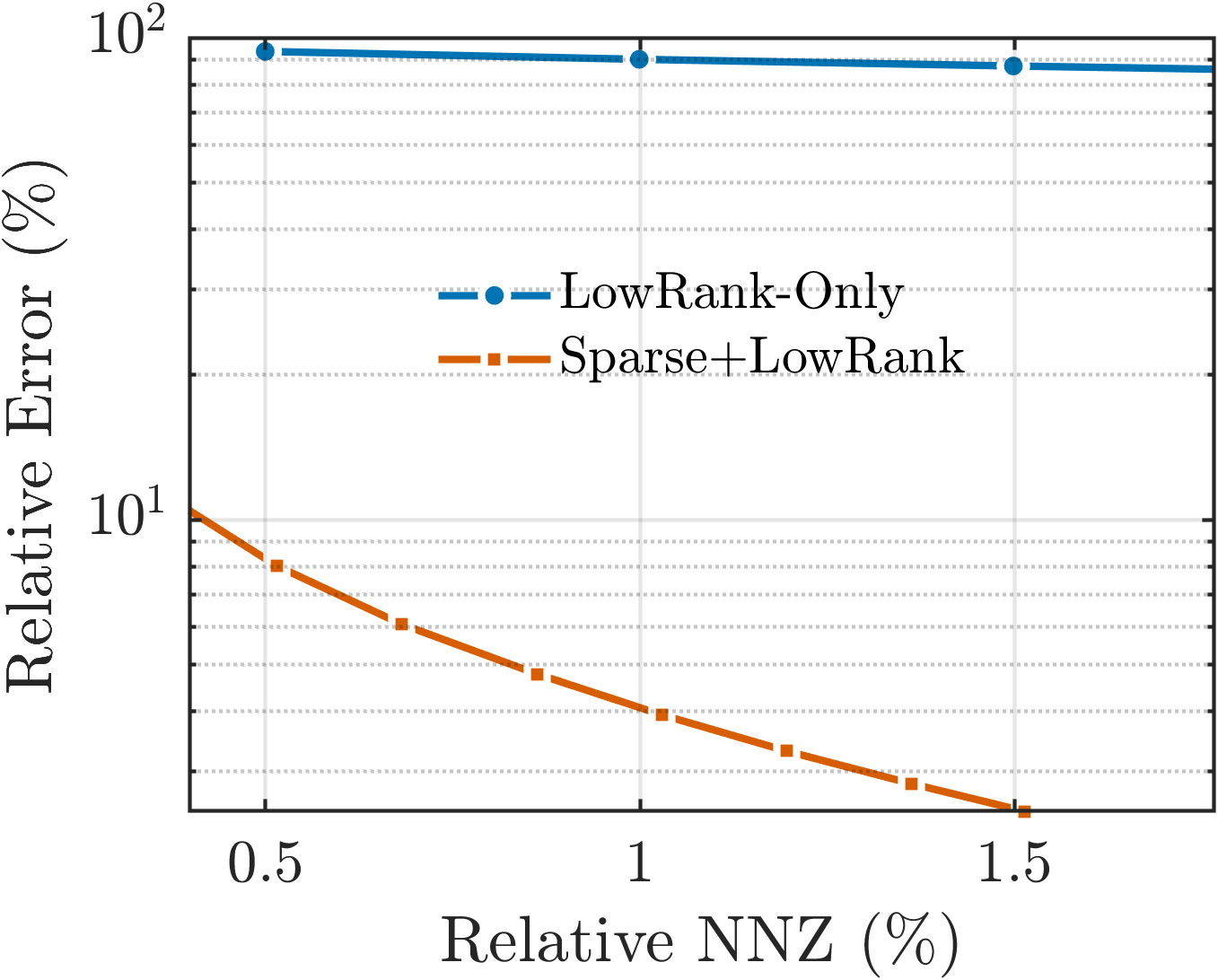}
			\caption{Motivating example for SLME using a dose-influence matrix $A$ from
				the treatment planning of a patient with lung cancer. Each panel plots the
				relative Frobenius-norm approximation error against the representation size
				relative to $\operatorname{nnz}(A)$. SLME is compared with a sparse-only
				approximation (left) and a low-rank-only approximation (right). Across all
				representation budgets shown, SLME achieves smaller approximation error
				than either alternative.}
			\label{fig:motivating_example}
		\end{figure}
		
		We formulate SLME as a bi-objective nonconvex optimization problem that
		explicitly balances approximation error against the computational cost of
		storing and applying the structured representation (hereafter, the \emph{representation cost}). Rather than prescribing a
		single trade-off parameter, the formulation seeks the Pareto frontier of
		accuracy--cost trade-offs. Each Pareto-optimal solution provides a potentially
		useful embedding corresponding to a different representation budget or desired
		level of approximation accuracy.
		
		To approximate this Pareto frontier efficiently, we develop a family of
		trust-region algorithms consisting of R-Trust, R2-Trust, and R3-Trust. The
		proposed algorithms generate decompositions spanning a range of representation
		budgets in a single run, without requiring a user-specified accuracy--cost
		trade-off parameter.\\\\
		\textbf{Contributions.}
		The main contributions of this work are as follows:
		
		\begin{itemize}[leftmargin=1.5em]
			
			\item \emph{Methodological.}
			We introduce SLME, a new perspective on sparse-plus-low-rank matrix
			decomposition aimed at matrix compression and efficient downstream computation
			rather than recovery of latent structure.
			
			\item \emph{Algorithmic.}
			We develop a family of fast and scalable trust-region algorithms (R3-Trust
			and its two related variants, R-Trust and R2-Trust) for approximating the
			Pareto frontier of the nonconvex SLME problem.
			
			\item \emph{Computational.}
			Using real dose-influence matrices from radiotherapy optimization, we
			demonstrate that R3-Trust is substantially faster and achieves more
			favorable accuracy--cost trade-offs than representative recovery-oriented methods, namely ADMM~\cite{boyd2011distributed}, GoDec~\cite{Zhou_undated-ur}, and VB~\cite{Babacan2012-gm}. On synthetic
			instances, we further show that SLME and R3-Trust can be applied to
			sparse-plus-low-rank recovery, where they compare favorably with existing
			methods in terms of reconstruction accuracy and computational time. An
			open-source MATLAB implementation of all proposed algorithms and
			experiments is available in the
			\href{https://github.com/Radiotherapy-Optimization/SLME-matrix-embedding}{SLME-matrix-embedding}
			GitHub repository.
			
		\end{itemize}
		The remainder of the paper is organized as follows. Section \ref{sec:background} reviews recovery-oriented sparse-plus-low-rank models and
		the benchmark methods used in our numerical comparisons. Section
		\ref{sec:slme} presents the SLME framework, its bi-objective formulation, and
		its relationship to matrix recovery. Section \ref{sec:algorithms} develops the proposed trust-region algorithms and establishes their theoretical properties. Section \ref{sec:numerics} reports numerical experiments on real radiotherapy matrices and synthetic recovery instances. Finally, Section \ref{sec:conclusion} summarizes the main findings and discusses directions for future research.

		\section{Background and benchmark methods}\label{sec:background}
		
		We briefly review recovery-oriented sparse-plus-low-rank models and the three classes of methods used as benchmarks in our numerical experiments. These
		methods differ from SLME in their primary objective: they seek to recover
		latent structure, whereas SLME seeks a computationally efficient representation
		of a given matrix. For a comprehensive review of recovery-oriented methods, we
		refer the reader to~\cite{Zhou2014-ug}.

		\subsection{Sparse-plus-low-rank recovery models}
		
		The classical starting point is \emph{principal component analysis} (PCA),
		which seeks a best rank-$k$ approximation of a data matrix:
		
		\begin{equation}\tag{PCA}\label{PCA}
			\min_{L,E} \ \|E\|_{\text{F}}^2
			\quad \text{s.t.} \quad
			\operatorname{rank}(L)\leq k,\qquad A=L+E.
		\end{equation}
		
		Here, $A\in\mathbb{R}^{m\times n}$ is the observed data matrix,
		$L\in\mathbb{R}^{m\times n}$ is its rank-$k$ approximation, and
		$E\in\mathbb{R}^{m\times n}$ represents small-magnitude noise. Although
		problem~\eqref{PCA} is nonconvex because of the rank constraint, a globally optimal solution can be computed efficiently by truncating the singular value decomposition of $A$~\cite{eckart1936approximation}. An insightful interpretation of PCA is that it can recover the underlying low-dimensional data $L$ from the observations in $A$, even when they are contaminated with random small noise $E$. However, this interpretation also reveals a weakness of PCA in handling gross errors, such as large-magnitude outliers in $E$. To address this problem, \emph{robust PCA} (RPCA) was introduced:
		
		\begin{equation}\tag{RPCA}\label{RPCA}
			\min_{L,S} \
			\operatorname{rank}(L)+\lambda\|S\|_0
			\quad \text{s.t.} \quad
			A=L+S,
		\end{equation}
		where $S$ represents sparse gross corruptions, $\|S\|_0$ denotes the number of nonzero entries in $S$, and $\lambda>0$ controls the trade-off between low rank and sparsity. Unlike PCA, formulation~\eqref{RPCA} is robust to sparse errors of large magnitude, but it does not explicitly account for the small, dense noise handled by PCA. This limitation was later addressed by stable RPCA formulations, which introduce an additional noise matrix $E$ and model the observations as $A=L+S+E$~\cite{Zhou2010-ji}. In contrast to PCA, problem~\eqref{RPCA} does not admit a solution through a single truncated singular value decomposition and is computationally challenging because both the rank and cardinality functions are nonconvex and discontinuous. Numerous methods have therefore been proposed for solving RPCA and its variants; see~\cite{Zhou2014-ug} and the references therein. We next review three major classes of methods and select one representative algorithm from each class for use in our numerical comparisons.

		\subsection{Benchmark solution methods}
		
		\textbf{Convex optimization techniques.}
		A widely studied convex relaxation of RPCA, known as \emph{principal component pursuit} (PCP), replaces the rank and cardinality functions with their convex surrogates, the nuclear norm and the $\ell_1$-norm, respectively:
		
		\begin{equation}\tag{PCP}\label{PCP}
			\min_{L,S} \ \|L\|_*+\lambda\|S\|_1
			\quad \text{s.t.} \quad
			A=L+S.
		\end{equation}
		
		Here, $\|L\|_*$ denotes the nuclear norm, defined as the sum of the singular values of $L$, and $\|S\|_1$ denotes the entrywise $\ell_1$-norm of $S$. The parameter $\lambda>0$ controls the trade-off between low rank and sparsity. Under suitable incoherence and sparsity assumptions, PCP exactly recovers the underlying low-rank and sparse components with high probability~\cite{Candes2009-fm,Chandrasekaran2011-vc}. To account for both sparse gross corruptions ($S$) and small dense noise ($E$), PCP was later extended to \emph{stable principal component pursuit} (SPCP)~\cite{Zhou2010-ji}:
		
		\begin{equation}\tag{SPCP}\label{SPCP}
			\min_{L,S,E} \ \|L\|_*+\lambda\|S\|_1+\frac{\mu}{2}\|E\|_{\text{F}}^2
			\quad \text{s.t.} \quad
			A=L+S+E,
		\end{equation}
		where $E$ represents small-magnitude dense noise, while $\lambda>0$ and $\mu>0$ balance low rank, sparsity, and data fidelity. Under the assumptions used in the corresponding recovery analyses, the common choice $\lambda=1/\sqrt{\max\{m,n\}}$ yields favorable theoretical guarantees for both PCP and SPCP~\cite{Candes2009-fm,Zhou2010-ji}. Moreover, when the entries of $E$ are independent Gaussian random variables with variance $\sigma^2$, a theoretically motivated choice is $\mu=1/(\sqrt{2\max\{m,n\}}\,\sigma)$~\cite{Zhou2010-ji}. In practice, however, both parameters are often treated as hyper-parameters and tuned to the data and application.
		
		The nuclear norm and the entrywise $\ell_1$-norm are convex but nonsmooth, and both admit efficiently computable proximal operators~\cite{Parikh2014-ft}. The proximal operator of the nuclear norm is singular-value thresholding~\cite{cai2010singular}, whereas that of the $\ell_1$-norm is elementwise soft thresholding. Consequently, PCP and SPCP are well suited to first-order splitting methods, including augmented Lagrangian methods and their variants, such as the \emph{alternating direction method of multipliers} (ADMM), inexact augmented Lagrangian methods, and schemes with adaptive penalty parameters~\cite{Candes2009-fm,boyd2011distributed,Aybat2015-fo,zarepisheh2018computation}. Proximal-gradient and accelerated proximal-gradient methods have also been developed for these models~\cite{Zhou2010-ji,nesterov2007gradient,beck2009fast}. Alternatively, nuclear-norm minimization can be represented through semidefinite programming~\cite{Chandrasekaran2011-vc,fazel2002matrix}, although general-purpose interior-point methods are typically impractical for large matrices because of their computational and memory requirements~\cite{liu2010interior}. Although different algorithms in this category (i.e., convex optimization techniques) may have different computational performance and some may be advantageous for a particular application, the convex property of problem~\eqref{SPCP} implies that all algorithms are expected to provide similar results. In this paper, we use ADMM \cite{boyd2011distributed} for comparison purposes.

		\textbf{Nonconvex optimization techniques.}
		A second class of methods addresses the rank and cardinality constraints directly rather than replacing them with convex surrogates. These methods can be broadly divided into factorization-based and projection-based approaches. \emph{Factorization-based} methods parameterize the low-rank component as $L=HW$ and alternately update the factors $H$ and $W$ together with the sparse component $S$. By avoiding the repeated singular value decompositions required by nuclear-norm minimization, these methods can scale to substantially larger matrices, although the resulting optimization problems are nonconvex and generally lack global optimality guarantees~\cite{Shen2014-tq,Zhou2013-tj,cabral2013unifying}. \emph{Projection-based} methods instead retain $L$ and $S$ as explicit variables and alternate between projections onto the sets of low-rank and sparse matrices. A representative method in this category is \emph{GoDec}~\cite{Zhou_undated-ur}, which solves
		
		\begin{equation}\tag{GoDec}\label{GoDec}
			\min_{L,S} \ \|A-L-S\|_{\text{F}}^2
			\quad \text{s.t.} \quad
			\operatorname{rank}(L)\leq r,\qquad \|S\|_0\leq k.
		\end{equation}
		
		GoDec alternates between a rank-$r$ projection of $A-S$, computed using a truncated singular value decomposition and optionally accelerated through randomized techniques, and a $k$-sparse projection of $A-L$, obtained by retaining its $k$ largest-magnitude entries. The method therefore requires two user-specified hyper-parameters: an upper bound $r$ on the rank of the low-rank component and an upper bound $k$ on the cardinality of the sparse component. Because GoDec was identified as one of the top-performing methods in the comparative study of~\cite{Zhou2014-ug}, we use it as our representative nonconvex baseline.
		
		\textbf{Nonoptimization techniques.}
		A third class of methods replaces a deterministic optimization formulation with a probabilistic Bayesian model and infers the low-rank and sparse components from their posterior distributions. Sparsity-promoting priors are imposed on the latent factors or singular values associated with the low-rank component and on the entries of the sparse component. The effective rank, sparsity pattern, and noise level can then be inferred from the data, substantially reducing the manual hyper-parameter tuning required by many optimization-based methods~\cite{ding2011bayesian,Babacan2012-gm,zhao2014robust}. Posterior inference may be performed using Markov chain Monte Carlo sampling~\cite{ding2011bayesian} or, more efficiently, variational Bayesian approximations~\cite{Babacan2012-gm,zhao2014robust}. In our numerical experiments, we use the variational Bayesian robust PCA method of Babacan et al.~\cite{Babacan2012-gm}, denoted by \emph{VB}, as the representative Bayesian baseline. This choice is motivated by the comparative study in~\cite{Zhou2014-ug}, in which VB was among the top-performing methods.

		\section{Sparse-plus-low-rank matrix embedding}\label{sec:slme}
		
		\textbf{Bi-objective formulation.}
		SLME seeks to balance approximation accuracy, measured by $\|A-(S+HW)\|_{\text{F}}^2$, against the cost of storing and applying the structured representation $(S,H,W)$. Assuming dense factors $H\in\mathbb{R}^{m\times r}$ and $W\in\mathbb{R}^{r\times n}$, this cost is proportional to $\operatorname{nnz}(S)+r(m+n)$. We refer to this quantity as the \emph{representation cost}: it measures the cost of storing the representation and of applying it in downstream computations, such as the matrix-vector products in~\eqref{RO-Problem}, and should not be confused with the computational cost of \emph{constructing} the representation, that is, the runtime of the algorithms developed in Section~\ref{sec:algorithms}. This leads naturally to the following bi-objective optimization problem: 
		
		\begin{equation} 
			\min_{H,W,S} \{\operatorname{nnz}(H,W,S),~\|E\|_{\text{F}}^2\}~~ \text{s.t.}~~ A = S+HW+E. 
		\end{equation}
		
		Let $L = HW$. Then, the representation cost of SLME, $\operatorname{nnz}(H,W,S)$, can be expressed as $(m+n)\times \operatorname{rank}(L) + \|S\|_0$. Therefore, we can formulate the SLME problem as follows:
		\begin{equation}\label{LSME} 
			\min_{H,W,S} \{(m+n)\times \operatorname{rank}(L)+\|S\|_0,~\|E\|_{\text{F}}^2\}~~ \text{s.t.}~~ A = S+L+E. 
		\end{equation}
		
		Problem~\eqref{LSME} seeks the Pareto frontier between representation cost and approximation accuracy. Each Pareto-optimal solution provides a potentially useful embedding corresponding to a different representation budget.\ 
		
		\textbf{Relation to matrix recovery}. Although SLME resembles sparse-plus-low-rank recovery models such as
		\eqref{RPCA}, \eqref{PCP}, and \eqref{SPCP}, the two settings differ in two fundamental respects:
		
		\begin{itemize}[leftmargin=1.5em]
			
			\item\emph{Objective.}
			In matrix recovery, the low-rank and sparse components are latent objects of interest. In contrast, SLME seeks a computationally efficient representation of a given matrix, without requiring the individual components to be interpretable or to correspond to an underlying data-generating structure.
			
			\item\emph{Model tuning.}
			Recovery formulations typically seek a particular decomposition and therefore require tuning parameters that determine the desired balance among rank, sparsity, and noise. SLME, by contrast, is parameter-free and generates a collection of Pareto-optimal embeddings spanning different levels of approximation accuracy and representation cost.
		\end{itemize}
		
		Consistent with these different objectives, $E$ represents measurement noise in
		recovery models, whereas in SLME, $E=A-L-S$ is an approximation residual whose
		magnitude quantifies the accuracy sacrificed for a lower representation cost.
		\section{Recursive trust-region algorithms}\label{sec:algorithms}
		This section develops the algorithmic core of the paper. We begin with
		R-Trust, a recursive algorithm that builds an SLME representation incrementally:
		at each iteration, it solves in closed form a subproblem that seeks the best
		sparse-plus-low-rank approximation of the current residual within a budget of
		at most $m+n$ additional stored values, a constraint we interpret as a trust
		region. We then introduce two generalizations: R2-Trust, which accommodates
		approximate sparse and low-rank projection subroutines, and R3-Trust, which
		instantiates these subroutines with randomized linear algebra and thereby
		scales to the large dense matrices arising in radiotherapy. The section
		concludes with a convergence analysis.
		
		\subsection{R-Trust}
		
		The design of R-Trust rests on three equivalent characterizations of the
		squared Frobenius norm of a matrix $A\in\mathbb{R}^{m\times n}$
		\cite{hogben2013handbook}:
		\begin{align}
			\|A\|^2_{\text{F}} &= \sum_{i=1}^{\min\{m,n\}} \sigma_{i}^{2} \label{Fro3} \\
			&= \sum_{i=1}^{m} \sum_{j=1}^{n} a_{ij}^{2} \label{Fro1} \\
			&= \operatorname{Tr} \left(A^TA\right), \label{Fro2}
		\end{align}
		where $a_{ij}$ denotes the $(i,j)$th entry of $A$ and $\sigma_i$ its $i$th
		largest singular value. The entrywise form \eqref{Fro1} quantifies the effect
		of removing individual entries, while the spectral form \eqref{Fro3}
		quantifies the effect of removing leading singular triplets $(u_i, \sigma_i, v_i)$; R-Trust exploits
		precisely this interplay.
		
		The first building block is the classical Eckart--Young--Mirsky
		theorem~\cite{eckart1936approximation,10.1093/qmath/11.1.50}, which
		identifies a best rank-$k$ approximation of a matrix and quantifies the
		resulting reduction in the squared Frobenius norm: subtracting this
		approximation removes exactly the sum of squares of the $k$ leading singular
		values.
		
		\begin{lemma}\label{lem1}
			Let $A\in\mathbb{R}^{m\times n}$ have singular values
			$\sigma_1\geq\sigma_2\geq\cdots\geq0$ and a singular value
			decomposition $A=\sum_{i=1}^{\min\{m,n\}}\sigma_i u_i v_i^T$, and let
			$1\leq k\leq\min\{m,n\}$.
			Then $L^*=\sum_{i=1}^{k}\sigma_i u_i v_i^T$ solves
			\[
			\begin{array}{ll}
				\mbox{minimize}	& \|A-L\|_{\text{F}}^2 \\
				\mbox{subject to}	& \operatorname{rank}(L)\leq k,
			\end{array}
			\]
			and the optimal objective value is
			$\|A\|^2_{\text{F}}-\sum _{i=1}^{k}\sigma_i^2$.
		\end{lemma}
		\begin{proof}
			By Mirsky's low-rank approximation theorem, the truncated singular value
			decomposition $L^*$ is a best rank-$k$ approximation of $A$ with respect to
			the Frobenius norm (indeed, with respect to any unitarily invariant
			norm)~\cite{10.1093/qmath/11.1.50}. It remains to compute the optimal
			objective value. Since
			$A-L^*=\sum_{i=k+1}^{\min\{m,n\}}\sigma_i u_i v_i^T$ is itself a singular
			value decomposition, applying \eqref{Fro3} to $A-L^*$ yields
			\[
			\|A-L^*\|^2_{\text{F}}
			=\sum _{i=k+1}^{\min\{m,n\}}\sigma _{i}^{2}
			=\|A\|^2_{\text{F}}-\sum _{i=1}^{k}\sigma_i^2,
			\]
			which completes the proof.
		\end{proof}
		
		The second building block is the sparse analogue of Lemma \ref{lem1}: the best
		approximation of $A$ by a $k$-sparse matrix, that is, a matrix with at most
		$k$ nonzero entries, together with the corresponding reduction in the squared
		Frobenius norm. For an index set
		$\Omega\subseteq\{1,\ldots,m\}\times\{1,\ldots,n\}$, let
		$\mathcal{P}_\Omega(A)$ denote the matrix that agrees with $A$ on $\Omega$
		and is zero elsewhere:
		\[
		\mathcal{P}_\Omega(A)_{ij} =
		\begin{cases}
			a_{ij} & \text{if $(i,j)\in\Omega$}, \\
			0 & \text{otherwise}.
		\end{cases}
		\]
		By \eqref{Fro1}, $\mathcal{P}_\Omega(A)$ is the unique minimizer of
		$\|A-S\|_{\text{F}}^2$ over all matrices $S$ with
		$\operatorname{supp}(S)\subseteq\Omega$. With a slight abuse of notation,
		for an integer $1\leq k\leq mn$, we also write $\mathcal{P}_{k}(A)$ for the
		matrix that agrees with $A$ on its $k$ largest-magnitude entries (with ties
		broken arbitrarily) and is zero elsewhere; that is,
		$\mathcal{P}_{k}(A)=\mathcal{P}_{\Omega_k}(A)$, where $\Omega_k$ collects the
		positions of the $k$ largest-magnitude entries of $A$.
		
		\begin{lemma}\label{lem3}
			Let $A\in\mathbb{R}^{m\times n}$ and $1\leq k\leq mn$. Then
			$S^*=\mathcal{P}_{k}(A)$ solves
			\[
			\begin{array}{ll}
				\mbox{minimize}	& \|A-S\|_{\text{F}}^2 \\
				\mbox{subject to}	& \|S\|_0\leq k,
			\end{array}
			\]
			and the optimal objective value is
			$\|A\|^2_{\text{F}}-\|\mathcal{P}_{k}(A)\|_{\text{F}}^2$.
		\end{lemma}
		\begin{proof}
			Let $S^*$ be an optimal solution. From \eqref{Fro1}, the objective function decomposes elementwise as $\sum_{i,j} (a_{ij} - s_{ij})^2$. To minimize this sum, any nonzero entry in $S^*$ must perfectly match the corresponding entry in $A$, implying $s^*_{ij} = a_{ij}$ for all $(i,j) \in \operatorname{supp}(S^*)$. 
			
			Thus, $S^* = \mathcal{P}_\Omega(A)$ for some support set $\Omega$ with cardinality $\operatorname{card}(\Omega) \leq k$. Substituting this into the objective yields:
			\begin{equation*}
				\begin{split}
					\|A-S^*\|^2_{\text{F}} &= \sum_{i=1}^{m}\sum_{j=1}^{n} (a_{ij} - s^*_{ij})^{2} \\
					&= \sum_{i=1}^{m}\sum_{j=1}^{n} a_{ij}^{2} - 2\sum_{i=1}^{m}\sum_{j=1}^{n} a_{ij}s^*_{ij} + \sum_{i=1}^{m}\sum_{j=1}^{n} {s^*_{ij}}^2 \\
					&= \|A\|^2_{\text{F}} - \|S^*\|^2_{\text{F}},
				\end{split}
			\end{equation*}
			where the last equality follows because $a_{ij}s^*_{ij} = {s^*_{ij}}^2$ for all $i,j$. To minimize the residual error, we must maximize $\|S^*\|^2_{\text{F}}$. This is achieved by choosing $\Omega$ to correspond to the indices of the $k$ largest entries of $A$ in absolute magnitude. Thus, $S^* = \mathcal{P}_k(A)$, establishing the desired result.
		\end{proof}
		
		Lemmas \ref{lem1} and \ref{lem3} price the two types of updates in a common currency.
		Appending a rank-one term $\sigma_1u_1v_1^T$ to the sparse-plus-low-rank representation
		\eqref{SLME-representation} adds $m+n$ stored values and reduces
		$\|A\|^2_{\text{F}}$ by $\sigma_1^2$, whereas transferring the $m+n$
		largest-magnitude entries of $A$ to the sparse component adds the same number
		of stored values and reduces $\|A\|^2_{\text{F}}$ by
		$\|\mathcal{P}_{m+n}(A)\|^2_{\text{F}}$. This suggests a greedy strategy:
		allocate a budget of $m+n$ additional stored values (the cost of exactly one
		rank-one term) and seek the best sparse-plus-low-rank approximation
		affordable within that budget. In analogy with classical trust-region
		methods~\cite{conn2000trust}, which restrict each step to a region where the
		subproblem can be solved reliably, this budget constraint defines a trust
		region within which the subproblem admits a closed-form solution: the budget
		affords either a single rank-one term or $m+n$ sparse entries, and the
		optimal choice is determined by comparing the two candidate reductions,
		$\sigma_1^2$ and $\|\mathcal{P}_{m+n}(A)\|^2_{\text{F}}$. The following
		theorem formalizes this observation.
		
		\begin{theorem}\label{thm:1}
			Let $\sigma_1$, $u_1$, and $v_1$ denote the leading singular value and the
			corresponding left and right singular vectors of $A\in\mathbb{R}^{m\times n}$.
			A solution to the optimization problem
			\begin{equation*}
				\begin{array}{ll}
					\mbox{minimize}	& \|A-(S+L)\|_{\text{F}}^2 \\
					\mbox{subject to}	& \|S\|_0+(m+n)\operatorname{rank}(L)\leq m+n
				\end{array}
			\end{equation*}
			is given by
			\[
			(S^*,L^*)=
			\begin{cases}
				(0,\ \sigma_1u_1v_1^T) & \text{if } \|\mathcal{P}_{m+n}(A)\|^2_{\text{F}}\leq\sigma_1^2,\\[2pt]
				(\mathcal{P}_{m+n}(A),\ 0) & \text{otherwise}.
			\end{cases}
			\]
		\end{theorem}
		\begin{proof}
			Since $\|S\|_0\geq0$ and $\operatorname{rank}(L)$ is a nonnegative integer,
			the constraint holds if and only if either
			\[
			L=0 \text{ and } \|S\|_0\leq m+n,
			\]
			or
			\[
			\operatorname{rank}(L)= 1 \text{ and } S=0.
			\]
			By Lemma \ref{lem3}, the best feasible point of the first type is
			$S=\mathcal{P}_{m+n}(A)$, with objective value
			$\|A\|^2_{\text{F}}-\|\mathcal{P}_{m+n}(A)\|^2_{\text{F}}$. By Lemma \ref{lem1} with
			$k=1$, the best feasible point of the second type is
			$L=\sigma_1u_1v_1^T$, with objective value
			$\|A\|^2_{\text{F}}-\sigma_1^2$. Comparing the two values yields the
			claimed solution.
		\end{proof}
		
		R-Trust, summarized in Algorithm \ref{alg:1}, applies Theorem \ref{thm:1} recursively. The
		algorithm maintains a residual $A_r$, initialized to $A$, which stores the
		part of the matrix not yet captured by the representation. At each iteration,
		it solves the trust-region subproblem of Theorem \ref{thm:1} for $A_r$ exactly: if
		the sparse candidate achieves the larger reduction, the support of
		$\mathcal{P}_{m+n}(A_r)$ is added to the accumulated support $\Omega$;
		otherwise, the leading singular triplet of $A_r$ is appended to the factors
		$H$ and $W$. In either case, the selected approximation is subtracted from
		$A_r$, and the procedure repeats until $\|A_r\|_{\text{F}}<\epsilon$. Each
		iteration thus enlarges the representation by at most $m+n$ stored values
		while achieving the largest possible reduction of the residual norm within
		that budget.
		
		It is worth mentioning that Theorem~\ref{thm:1} guarantees global
		optimality of each recursive subproblem within the specified trust region. It does not, by itself, imply that
		the complete sequence is globally optimal for the original bi-objective SLME problem.
		\begin{algorithm}[H]
			\caption{Recursive Trust-Region (R-Trust)}\label{alg:1}
			\begin{algorithmic}
				\Require $\epsilon>0$
				\Ensure $\|A - (S + L)\|_{\text{F}} < \epsilon$
				\State $A_r \gets A$
				\State $H \gets \emptyset$
				\State $W \gets \emptyset$
				\State $\Omega \gets \emptyset$
				\While{$\|A_r\|_{\text{F}} \geq \epsilon$}
				\State $(u, \sigma, v) \gets \operatorname{SVD}(A_r,1)$
				\If{$\|\mathcal{P}_{m+n}(A_r)\|_{\text{F}}^2 > \sigma^2$} \Comment{Theorem \ref{thm:1}}
				\State $\Omega \gets \Omega\cup \operatorname{supp}\big(\mathcal{P}_{m+n}(A_r)\big)$ \Comment{Remark \ref{remark:1}}
				\State $A_r \gets A_r - \mathcal{P}_\Omega(A_r)$
				\Else
				\State $H \gets \operatorname{addCol}(H, \sigma u)$
				\State {\footnotesize $\triangleright$ $\operatorname{addCol}(H, \sigma u)$ appends $\sigma u$ to $H$ as a new last column}
				\State $W \gets \operatorname{addRow}(W, v^T)$
				\State {\footnotesize $\triangleright$ $\operatorname{addRow}(W, v^T)$ appends $v^T$ to $W$ as a new last row}
				\State $A_r \gets A_r - \sigma u v^T$
				\EndIf
				\EndWhile
				\State $L \gets H W$
				\State $S \gets \mathcal{P}_\Omega(A - L)$
			\end{algorithmic}
		\end{algorithm}
		
		\begin{remark}\label{remark:1}
			In the sparse step of Algorithm \ref{alg:1}, the residual is updated by subtracting
			$\mathcal{P}_\Omega(A_r)$ over the entire accumulated support $\Omega$,
			rather than only over the newly selected positions
			$\operatorname{supp}\big(\mathcal{P}_{m+n}(A_r)\big)$. This comes at no
			cost: the positions in $\Omega$ are already counted in $\|S\|_0$, so
			zeroing the residual there does not increase the size of the
			representation, while by \eqref{Fro1} it can only decrease
			$\|A_r\|_{\text{F}}$. Indeed, entries at previously selected positions may
			become nonzero again after intermediate low-rank updates, and removing
			them anew accelerates the reduction of the residual for free.
		\end{remark}
		
		\subsection{R2-Trust}
		
		Each iteration of R-Trust requires the leading singular triplet of the
		residual and its $m+n$ largest-magnitude entries. Both computations are exact but can become expensive for large matrices. In this subsection, we
		generalize R-Trust so that both may be replaced by approximate projection
		subroutines.

		While the development of R-Trust relied heavily on the singular value and elementwise definitions of the Frobenius norm \eqref{Fro3} and \eqref{Fro1}, we now leverage its trace definition \eqref{Fro2}. This perspective reveals that if the right factor $W$ of a low-rank approximation $L=HW$ is fixed and has orthonormal rows, the optimal left factor $H$ can be identified via a simple closed-form solution. This provides a powerful mechanism for ``rectifying'' previously selected low-rank components at each iteration. More importantly, it opens the door to a generalized framework, Rectified R-Trust (R2-Trust), where the exact SVD can be replaced by any low-rank projection to improve the computational efficiency.

		\begin{theorem}\label{thm:2}
			Let $A\in\mathbb{R}^{m\times n}$, and let $W\in\mathbb{R}^{r\times n}$ have
			orthonormal rows, that is, $WW^T=I$. Then
			\begin{equation}\label{eq:1}
				\arg\min_{H\in\mathbb{R}^{m\times r}}\|A - H W\|_{\text{F}}^2 = A W^T,
			\end{equation}
			and
			\begin{equation}\label{eq:2}
				\min_{H\in\mathbb{R}^{m\times r}}\|A - H W\|_{\text{F}}^2 = \|A\|_{\text{F}}^2 - \|A W^T\|_{\text{F}}^2.
			\end{equation}
		\end{theorem}
		\begin{proof}
			Choose $\bar{W}\in\mathbb{R}^{(n-r)\times n}$ whose rows extend the rows of
			$W$ to an orthonormal basis of $\mathbb{R}^n$, so that the stacked matrix
			$\tilde{W}=\begin{bmatrix} W\\ \bar{W}\end{bmatrix}$ is orthogonal; in
			particular, $\tilde{W}^T\tilde{W}=W^TW+\bar{W}^T\bar{W}=I$ and
			$W\bar{W}^T=0$. For any $B\in\mathbb{R}^{m\times n}$, \eqref{Fro2} gives
			\[
			\|B\|_{\text{F}}^2
			=\operatorname{Tr}\big(B\tilde{W}^T\tilde{W}B^T\big)
			=\|B\tilde{W}^T\|_{\text{F}}^2
			=\|BW^T\|_{\text{F}}^2+\|B\bar{W}^T\|_{\text{F}}^2.
			\]
			Applying this identity to $B=A-HW$, and using $WW^T=I$ and
			$W\bar{W}^T=0$, we obtain
			\[
			\|A-HW\|_{\text{F}}^2
			=\|AW^T-H\|_{\text{F}}^2+\|A\bar{W}^T\|_{\text{F}}^2.
			\]
			The second term does not depend on $H$, and the first term vanishes
			precisely when $H=AW^T$. Hence $H^*=AW^T$ is the unique minimizer, and the
			optimal objective value is
			$\|A\bar{W}^T\|_{\text{F}}^2=\|A\|_{\text{F}}^2-\|AW^T\|_{\text{F}}^2$.
		\end{proof}

		Theorem \ref{thm:2} has two algorithmic consequences for Algorithm
		\ref{alg:1}. First, whenever the residual $A_r$ and the factor $W$ are at
		hand, the optimal companion factor is available in closed form as
		$H=A_rW^T$. This suggests appending a correction step to each iteration that
		readjusts $H$ to the most recent $W$. Second, by \eqref{eq:2}, the quantity
		$\|A_rW^T\|^2_{\text{F}}$ equals the reduction in the squared residual norm
		achieved by fitting the low-rank component on the rows of $W$; crucially,
		this measure of progress remains valid even when $W$ has not been obtained
		from an exact singular value decomposition, so it can replace the singular
		value $\sigma$ used in Algorithm \ref{alg:1}.
		
		These observations allow us to generalize Algorithm \ref{alg:1} in two
		directions. The exact rank-one SVD may be replaced by any subroutine
		$\operatorname{lowRankProjection}$ that takes the residual $A_r$ and returns
		an approximate solution of the problem in Lemma \ref{lem1}, under the sole
		assumption that the returned factor has orthonormal rows orthogonal to the
		kernel of $A_r$; Corollary \ref{remark:2} below shows that this assumption
		guarantees that the accumulated factor $W$ retains orthonormal rows, as
		required by Theorem \ref{thm:2}. Likewise, the exact selection of the
		largest-magnitude entries may be replaced by any subroutine
		$\operatorname{sparseProjection}$ that returns an approximate solution of the
		problem in Lemma \ref{lem3}. Both subroutines take a common parameter
		$k\geq1$, referred to as the \emph{batch size}, which scales the trust
		region: $\operatorname{lowRankProjection}(A_r,k)$ returns the right factor of
		an approximate rank-$k$ projection of the residual, and
		$\operatorname{sparseProjection}(A_r,k)$ returns the support of approximately
		$k(m+n)$ largest-magnitude entries. One iteration may therefore append up to
		$k(m+n)$ stored values rather than the $m+n$ of Algorithm~\ref{alg:1}, which
		corresponds to $k=1$. The resulting method, R2-Trust, is summarized in
		Algorithm \ref{alg:rect}. At each iteration, it forms both candidate updates,
		accepts the one that achieves the larger reduction of the residual norm
		relative to the number of stored values it adds, and then rectifies $H$
		through \eqref{eq:1}. 
		
		\begin{algorithm}[H]
			\caption{Rectified R-Trust (R2-Trust)}\label{alg:rect}
			\begin{algorithmic}
				\Require $\epsilon>0$; batch size $k\geq1$; $\operatorname{lowRankProjection}(A_r,k)$ has orthonormal rows which are orthogonal to the kernel of $A_r$
				\Ensure $\|A - (S + L)\|_{\text{F}} < \epsilon$
				\State $A_r \gets A$
				\State $S \gets 0$
				\State $W \gets \emptyset$
				\State $\Omega \gets \emptyset$
				\While{$\|A_r\|_{\text{F}} \geq \epsilon$}
				\State $\Omega_r \gets \operatorname{sparseProjection}(A_r, k)$ \Comment{$\operatorname{supp}\big(\mathcal{P}_{m+n}(A_r)\big)$ in Algorithm \ref{alg:1}}
				\State {\footnotesize $\triangleright$ returns the support of the $\approx k(m+n)$ largest-magnitude entries of $A_r$}
				\State $\tilde{\Omega} \gets \Omega\cup \Omega_r$
				\State $W_r \gets \operatorname{lowRankProjection}(A_r, k)$ \Comment{$\operatorname{SVD}(A_r,1)$ in Algorithm \ref{alg:1}}
				\State {\footnotesize $\triangleright$ returns the right factor ($k$ orthonormal rows) of an approximate rank-$k$ projection}
				\State $\tilde{W} \gets \operatorname{addRow}(W, W_r)$
				\If{$\dfrac{\|\mathcal{P}_{\tilde{\Omega}}(A_r)\|_{\text{F}}^2}{\operatorname{card}(\Omega_r)} > \dfrac{\|A_r\tilde{W}^T\|_{\text{F}}^2}{(m+n)\operatorname{rank}(W_r)}$} \Comment{Theorem \ref{thm:2} \eqref{eq:2}}
				\State $\Omega \gets \tilde{\Omega}$
				\State $S \gets S + \mathcal{P}_{\Omega}(A_r)$
				\State $A_r \gets A - S$
				\State $ W \gets \operatorname{lowRankProjection}(A_r,\operatorname{rank}(W))$
				\Else
				\State $W \gets \tilde{W}$
				\EndIf
				\State $H_r \gets A_r W^T$ \Comment{Theorem \ref{thm:2} \eqref{eq:1}}
				\State $A_r \gets A_r-H_r W$ \Comment{$A_r W^T = 0$ by Corollary \ref{remark:2}}
				\EndWhile
				\State $A_r \gets A - S$
				\State $ W \gets \operatorname{lowRankProjection}(A_r,\operatorname{rank}(W))$
				\State $H \gets (A - S) W^T$
				\State $L \gets H W$
				\State $S \gets \operatorname{sparseProjection}(A-L,\Omega)$ \Comment{$\mathcal{P}_\Omega(A - L)$ in Algorithm \ref{alg:1}}
			\end{algorithmic}
		\end{algorithm}
		
		\begin{corollary}\label{remark:2}
			Throughout Algorithm \ref{alg:rect}, the rows of $W$ remain orthonormal,
			and the residual at the end of each iteration satisfies $A_rW^T=0$.
		\end{corollary}
		\begin{proof}
			Each iteration ends by setting $H_r=A_rW^T$ and replacing $A_r$ with
			$A_r-H_rW$. Since $WW^T=I$,
			\[
			(A_r - H_r W) W^T = A_r W^T - A_r W^T\big(W W^T\big) = A_r W^T - A_r W^T = 0,
			\]
			so every row of $W$ lies in the kernel of the updated residual. At the next
			iteration, the rows returned by $\operatorname{lowRankProjection}$ are, by
			assumption, orthonormal and orthogonal to this kernel; in particular, they
			are orthogonal to all previous rows of $W$, so the augmented factor
			$\tilde{W}$ again has orthonormal rows. 
		\end{proof}
		
		\begin{remark}\label{rem:orthogonality}
			In practice, standard rank-revealing factorizations, such as the exact SVD or Randomized SVD (RSVD), naturally generate right singular vectors that are inherently orthonormal and strictly orthogonal to the kernel of the residual. Consequently, they satisfy the conditions of Corollary \ref{remark:2} by default. However, if an alternative low-rank projection heuristic is employed to prioritize computational efficiency and does not implicitly guarantee orthogonality to the kernel of $A_r$, the new rows $W_r$ may overlap with the existing basis. In such cases, one must explicitly orthogonalize the new rows against the accumulated basis $W$. This is achieved by applying a Gram--Schmidt projection, $W_r \gets W_r - (W_r W^T)W$, followed by orthonormalization of the rows (e.g., via QR), prior to augmentation. This correction ensures that the augmented matrix $\tilde{W}$ rigorously maintains row orthonormality without compromising the algorithmic framework.
		\end{remark}
		
		\subsection{R3-Trust}
		
		Whereas R2-Trust is a general framework that accepts any pair of projection
		subroutines satisfying its assumptions, R3-Trust is a concrete instantiation
		of this framework based on randomized linear algebra: the low-rank projection
		is carried out by a randomized SVD, and the sparse projection by an
		approximate quantile threshold.
		
		Consider first the sparse projection. The exact projection selects the
		$k(m+n)$ largest-magnitude entries of the residual, which requires sorting
		all $mn$ entries. Its support, however, is completely characterized by a
		single cutoff value (the $k(m+n)$th largest magnitude) which can be
		estimated cheaply as an empirical quantile. Suppose that $Q$ is an
		approximate quantile function such that, for all $0\leq x\leq 1$,
		\[
		\frac{\operatorname{card}\big(\{(i,j)\mid |A_{ij}| \geq Q(A,x) \}\big)}{mn}\approx x.
		\]
		Then $q=Q\big(A_r, \frac{k(m+n)}{mn}\big)$ serves as a cutoff threshold, and
		the sparse projection returns the index set
		$\Omega_r=\{(i,j)\mid |{(A_r)}_{ij}| \geq q\}$, which contains approximately
		$k(m+n)$ entries. In our implementation, $Q$ is estimated from a small random
		sample of the entries of $A_r$, so the cost of the sparse projection reduces
		to a single pass that compares each entry against $q$.
		
		Consider next the low-rank projection. The exact truncated SVD is replaced by
		a randomized SVD~\cite{doi:10.1137/090771806}, which computes an approximate
		rank-$k$ factorization using only a small number of passes over the residual.
		The right factor $V$ returned by $\operatorname{RSVD}(A_r,k)$ has orthonormal
		columns that lie in the row space of $A_r$; consequently, $W_r=V^T$ has
		orthonormal rows orthogonal to the kernel of $A_r$, exactly as Algorithm
		\ref{alg:rect} requires.
		
		The resulting method, R3-Trust, is summarized in Algorithm \ref{alg:rand}. It
		follows Algorithm \ref{alg:rect} with $\operatorname{sparseProjection}$
		realized by the sampled quantile threshold,
		$\operatorname{lowRankProjection}$ realized by the randomized SVD with batch
		rank $k$, and the rectification step that refits the entire low-rank factor by
		a randomized SVD. The batch size $k$, introduced with Algorithm~\ref{alg:rect}, remains the main parameter of the algorithm and
		controls the rank appended per low-rank update.
		\begin{algorithm}[H]
			\caption{Randomized R2-Trust (R3-Trust)}\label{alg:rand}
			\begin{algorithmic}
				\Require $\epsilon>0$; batch size $k\geq1$
				\Ensure $\|A - (S + L)\|_{\text{F}} < \epsilon$
				\State $A_r \gets A$
				\State $S \gets 0$
				\State $W \gets \emptyset$
				\State $\Omega \gets \emptyset$
				\While{$\|A_r\|_{\text{F}} \geq \epsilon$}
				\State $q \gets Q(A_r, \dfrac{k(m+n)}{mn})$ \Comment{Quantile estimated from a sample of $A_r$ entries}
				\State $\Omega_r \gets \{(i,j)\mid |{(A_r)}_{ij}| \geq q\}$ \Comment{$\operatorname{sparseProjection}(A_r, k)$ in Algorithm \ref{alg:rect}}
				\State $\tilde{\Omega} \gets \Omega\cup \Omega_r$ 
				\State $V \gets \operatorname{RSVD}(A_r,k)$ \Comment{Randomized SVD, $\operatorname{lowRankProjection}$ in Algorithm \ref{alg:rect}}
				\State $W_r \gets V^T$ 
				\State $\tilde{W} \gets \operatorname{addRow}(W, W_r)$ 
				\If{$\dfrac{\|\mathcal{P}_{\tilde{\Omega}}(A_r)\|_{\text{F}}^2}{\operatorname{card}(\Omega_r)} > \dfrac{\|A_r\tilde{W}^T\|_{\text{F}}^2}{k(m+n)}$}
				\State $\Omega \gets \tilde{\Omega}$
				\State $S \gets S + \mathcal{P}_{\Omega}(A_r)$
				\State $A_r \gets A - S$ 
				\State $ W \gets \operatorname{RSVD}(A_r,\operatorname{rank}(W))$ 
				\Else
				\State $W \gets \tilde{W}$
				\EndIf
				\State $H_r \gets A_r W^T$
				\State $A_r \gets A_r-H_r W$
				\EndWhile
				\State $H \gets (A-S) W^T$
				\State $L \gets H W$
				\State $S \gets \mathcal{P}_\Omega(A - L)$
			\end{algorithmic}
		\end{algorithm}
		
		\begin{remark}[per-iteration complexity]\label{rem:complexity}
			These replacements also lower the dominant per-iteration costs. Computing
			an exact truncated SVD of a dense residual by standard dense methods
			requires $\mathcal{O}(mnp)$ operations with
			$p=\min\{m,n\}$~\cite{golub2013matrix}, whereas the randomized SVD with
			batch rank $k$ requires only $\mathcal{O}(mnk)$~\cite{doi:10.1137/090771806}.
			Likewise, the exact sparse projection sorts all $mn$ residual entries at a
			cost of $\mathcal{O}(mn\log(mn))$, whereas the sampled quantile threshold
			sorts only a sample of $s\ll mn$ entries and then compares each entry
			against the resulting cutoff in a single pass, at a cost of
			$\mathcal{O}(mn+s\log s)$. Since the SVD dominates, for square matrices
			($m=n$) and a fixed batch size $k$, the per-iteration cost drops from
			cubic, $\mathcal{O}(n^3)$ for R-Trust, to quadratic, $\mathcal{O}(n^2)$
			for R3-Trust, consistent with the order-of-magnitude speedups observed in
			Section~\ref{sec:numerics}.
		\end{remark}
		
		\subsection{Convergence analysis}\label{subsec:convergence}
		
		Before detailing our convergence analysis, it is instructive to contextualize it within the broader landscape of sparse-plus-low-rank matrix approximation. Rigorous, global convergence guarantees in this domain are notoriously challenging to establish, and the available theoretical results remain limited. While convex relaxation approaches benefit from global convergence, their associated first-order solvers typically suffer from sublinear worst-case rates, such as $\mathcal{O}(1/t)$ for standard proximal-gradient schemes and $\mathcal{O}(1/t^2)$ for accelerated variants \cite{beck2009fast,Zhou2014-ug}. Conversely, theoretical guarantees for nonconvex formulations are often strictly local. For instance, GoDec provides one of the few available analyses, proving monotonic convergence of the objective to a local minimum alongside an asymptotic, locally linear rate; however, this rate depends on the strictly unknown geometric properties of the limit point \cite{Zhou_undated-ur}. Similarly, Bayesian methods only guarantee the monotonic improvement of a variational lower bound \cite{Babacan2012-gm}.
		
		The greedy, recursive nature of our proposed trust-region algorithms introduces comparable analytical complexities. Consequently, the theoretical results presented in this section should be interpreted in this spirit: they provide meaningful but partial guarantees, establishing a baseline rate, proving its worst-case tightness, and demonstrating stronger rates under specific structural assumptions. As is standard across the literature for this class of problems, a gap remains between these worst-case theoretical bounds and the highly efficient behavior observed in practice. The numerical
		experiments in Section \ref{sec:numerics} complement the analysis by comparing the
		empirical performance of all methods considered.
		
		Throughout this subsection, let $p=\min\{m,n\}$, let $\sigma_i(A)$ denote the
		$i$th largest singular value of $A$, and assume that the batch size satisfies
		$1\leq k\leq p$ and $k(m+n)\leq mn$. We measure progress by the fraction of
		the squared Frobenius norm of the current residual removed in one iteration,
		and we analyze the exact variants of the algorithms, namely R-Trust and
		R2-Trust with exact projection subroutines and without the rectification step that updates the previously selected low-rank components.
		The analysis conveys four main messages. First, the exact rank-$k$ and
		$k(m+n)$-sparse updates each remove at least a $k/p$ fraction of the squared
		residual norm, which implies linear convergence. Second, this dimension
		dependence is tight up to constant factors: for square matrices, a Hadamard
		example shows that the largest possible decrease within the trust region is
		$\Theta(k/n)$, and hence $\Theta(1/n)$ for R-Trust, where $k=1$. Third, when
		the current residual is entrywise nonnegative, one R-Trust iteration removes
		at least a $1/\sqrt{p}$ fraction of its squared norm; in particular, this
		guarantee applies to the first iteration for nonnegative dose-influence
		matrices. Fourth, when the residual admits a sparse-plus-low-rank decomposition
		whose components do not substantially mask one another, the guaranteed
		per-iteration decrease can be considerably stronger and may become independent
		of the matrix dimensions.

		\begin{proposition}[baseline guarantee]\label{prop:baseline}
			Let $A\in\mathbb{R}^{m\times n}$ be nonzero and $1\leq k\leq p$. Then
			\begin{equation}\label{eq:baseline}
				\sum_{i=1}^{k}\sigma_i(A)^2 \geq \frac{k}{p}\,\|A\|^2_{\text{F}}
				\qquad\text{and}\qquad
				\|\mathcal{P}_{k(m+n)}(A)\|^2_{\text{F}} \geq \frac{k(m+n)}{mn}\,\|A\|^2_{\text{F}} \geq \frac{k}{p}\,\|A\|^2_{\text{F}}.
			\end{equation}
		\end{proposition}
		\begin{proof}
			By \eqref{Fro3}, $\|A\|^2_{\text{F}}$ is the sum of the $p$ values
			$\sigma_1(A)^2\geq\cdots\geq\sigma_p(A)^2$, and the $k$ largest of them
			carry at least the fraction $k/p$ of the total. Similarly, by \eqref{Fro1},
			$\|A\|^2_{\text{F}}$ is the sum of the $mn$ squared entries, and the
			$k(m+n)$ largest of them carry at least the fraction $k(m+n)/(mn)$ of the
			total. Finally, $k(m+n)/(mn)=k(1/n+1/m)\geq k/p$.
		\end{proof}
		
		Proposition \ref{prop:baseline} applies verbatim to the two ``naive''
		single-component baselines: repeated entrywise thresholding (sparse-only) and
		the truncated SVD (low-rank-only) each remove at least a $k/p$ fraction of
		the squared norm per $k(m+n)$ stored values. Because every iteration of the
		proposed algorithms applies the better of the two elementary updates, the
		same guarantee yields linear convergence.
		
		\begin{corollary}[linear convergence]\label{cor:linear}
			Let $A^{(t)}_r$ denote the residual after $t$ iterations of Algorithm
			\ref{alg:1}. Then
			\[
			\|A^{(t)}_r\|^2_{\text{F}} \leq \Big(1-\frac{1}{p}\Big)^{t}\,\|A\|^2_{\text{F}},
			\]
			so the stopping criterion $\|A_r\|_{\text{F}}<\epsilon$ is
			met after at most $\lceil 2p\ln(\|A\|_{\text{F}}/\epsilon)\rceil$
			iterations for any $0<\epsilon<\|A\|_{\text{F}}$. For Algorithm
			\ref{alg:rect} with exact projections and without the rectification step, the
			same statements hold with $1/p$ replaced by $k/p$ and the iteration bound
			by $\lceil (2p/k)\ln(\|A\|_{\text{F}}/\epsilon)\rceil$.
		\end{corollary}
		\begin{proof}
			By Theorem \ref{thm:1}, one iteration of Algorithm \ref{alg:1} removes
			from the squared residual norm at least the larger of $\sigma_1(A_r)^2$
			and $\|\mathcal{P}_{m+n}(A_r)\|^2_{\text{F}}$, which by Proposition \ref{prop:baseline}
			with $k=1$ is at least $\|A_r\|^2_{\text{F}}/p$. Iterating the resulting
			contraction and using $(1-1/p)^t\leq e^{-t/p}$ gives the claims. In
			Algorithm \ref{alg:rect} with exact projections, the sparse branch removes
			$\|\mathcal{P}_{\tilde\Omega}(A_r)\|^2_{\text{F}}$, which is at least
			$\|\mathcal{P}_{k(m+n)}(A_r)\|^2_{\text{F}}$,
			while, by Theorem \ref{thm:2}, the low-rank branch removes
			$\|A_r\tilde{W}^T\|^2_{\text{F}}=\sum_{i=1}^{k}\sigma_i(A_r)^2$;
			by Proposition \ref{prop:baseline}, each of these quantities is at least
			$(k/p)\|A_r\|^2_{\text{F}}$, so the contraction holds regardless of which
			branch is taken.
		\end{proof}
		
		The next example shows that, without further assumptions, the baseline of
		Proposition \ref{prop:baseline} is tight.
		
		\begin{example}[worst-case tightness]\label{ex:hadamard}
			Let $n$ be an order for which a Hadamard matrix $H\in\{-1,1\}^{n\times n}$
			exists (for instance, any power of two, by Sylvester's construction), so
			that $H H^T =nI$. Then $\|H\|^2_{\text{F}}=n^2$ and all $n$ singular values
			of $H$ equal $\sqrt{n}$. Consequently, the best rank-one update removes
			$\sigma_1(H)^2=n=\frac{1}{n}\|H\|^2_{\text{F}}$, while, since every entry
			of $H$ has magnitude one, removing any $m+n=2n$ entries reduces
			$\|H\|^2_{\text{F}}$ by exactly $2n=\frac{2}{n}\|H\|^2_{\text{F}}$.
			Hence the optimal value of the trust-region subproblem of Theorem
			\ref{thm:1} is $\|H\|^2_{\text{F}}-2n$: no feasible update removes more
			than a $2/n$ fraction of the squared norm, regardless of how the budget is
			split between the sparse and low-rank components.
		\end{example}
		
		Improvements over the baseline are therefore only possible for structured
		inputs. We identify two such structures, both relevant to our application.
		The first is entrywise nonnegativity, which dose-influence matrices satisfy.
		
		\begin{proposition}[nonnegative matrices]\label{prop:nonneg}
			Let $A\in\mathbb{R}^{m\times n}$ be nonzero and entrywise nonnegative.
			Then
			\[
			\max\big\{\sigma_1(A)^2,\ \|\mathcal{P}_{m+n}(A)\|^2_{\text{F}}\big\}
			\;\geq\; \frac{1}{\sqrt{n}}\,\|A\|^2_{\text{F}}.
			\]
			In particular, one iteration of Algorithm \ref{alg:1} applied to a
			nonnegative residual removes at least a $1/\sqrt{n}$ fraction of the
			squared residual norm.
		\end{proposition}
		\begin{proof}
			Let $a_i$ denote the $i$th row of $A$. We bound the two candidate
			reductions from below through two simple surrogates. For the low-rank
			candidate, take the unit vector $w=\frac{1}{\sqrt{n}}\mathbf{1}_n$; since
			the entries of $A$ are nonnegative, the $i$th entry of $Aw$ equals
			$\frac{1}{\sqrt{n}}\|a_i\|_1$, and therefore
			\[
			\sigma_1(A)^2 \;\geq\; \|Aw\|_2^2 \;=\; \frac{1}{n}\sum_{i=1}^{m}\|a_i\|_1^2 .
			\]
			For the sparse candidate, let $\Omega'$ collect the position of one
			largest entry in each row, so that $\operatorname{card}(\Omega')=m\leq m+n$;
			since $\mathcal{P}_{m+n}(A)$ maximizes $\|\mathcal{P}_{\Omega}(A)\|_{\text{F}}$
			over all index sets of cardinality at most $m+n$ (Lemma \ref{lem3}),
			\[
			\|\mathcal{P}_{m+n}(A)\|^2_{\text{F}} \;\geq\; \|\mathcal{P}_{\Omega'}(A)\|^2_{\text{F}}
			\;=\; \sum_{i=1}^{m}\|a_i\|_\infty^2 .
			\]
			Now, by H\"older's inequality applied to each row,
			$\|a_i\|_2^2\leq\|a_i\|_1\|a_i\|_\infty$, and by the Cauchy--Schwarz
			inequality,
			\[
			\|A\|^2_{\text{F}}
			= \sum_{i=1}^{m}\|a_i\|_2^2
			\leq \sum_{i=1}^{m}\|a_i\|_1\|a_i\|_\infty
			\leq \sqrt{n}\,
			\Big(\frac{1}{n}\sum_{i=1}^{m}\|a_i\|_1^2\Big)^{\!1/2}
			\Big(\sum_{i=1}^{m}\|a_i\|_\infty^2\Big)^{\!1/2}.
			\]
			Since the geometric mean of two nonnegative numbers never exceeds their
			maximum, the right-hand side is at most
			$\sqrt{n}\,\max\big\{\frac{1}{n}\sum_i\|a_i\|_1^2,\ \sum_i\|a_i\|_\infty^2\big\}$,
			and the claim follows from the two surrogate bounds above.
		\end{proof}
		
		\begin{remark}\label{rem:nonneg}
			Applying Proposition \ref{prop:nonneg} to $A^T$ shows that $\sqrt{n}$ may
			be replaced by $\sqrt{p}$. The guarantee applies to every iteration whose
			current residual is entrywise nonnegative, in particular the first, since
			subsequent residuals need not remain nonnegative. Even as a one-iteration
			statement it is meaningful: it shows that alternating between the two
			update types can be quadratically better than committing to either one.
		\end{remark}
		
		The second structure is the one motivating this paper: the input is (close
		to) the sum of a sparse and a low-rank matrix. The next theorem shows that,
		under such an assumption, the per-iteration reduction improves far beyond the
		baseline, provided the two components do not mask each other. The masking is
		quantified by the fraction $\beta$ of the low-rank component's energy that is
		hidden on the support of the sparse component; small $\beta$ thus serves as
		our measure of incoherence between the two components.
		
		\begin{theorem}[sparse-plus-low-rank matrices]\label{thm:structured}
			Let $A=S+L$ with $\Omega=\operatorname{supp}(S)$, $\nu=\|S\|_0$,
			$\rho=\operatorname{rank}(L)$, and
			$\beta=\|\mathcal{P}_\Omega(L)\|^2_{\text{F}}/\|L\|^2_{\text{F}}$.
			Assume that
			\[
			1\leq k\leq \min\Big\{\frac{\nu}{m+n},\,\rho\Big\}
			\qquad\text{and}\qquad
			\beta<\frac{k}{\rho},
			\]
			and define
			\begin{equation}\label{eq:alpha-star}
				\alpha^* \;=\; \frac{\nu}{m+n}\Bigg[\,1+(1-\beta)
				\Bigg(\frac{1+\sqrt{k(m+n)/\nu}}{\sqrt{k/\rho}-\sqrt{\beta}}\Bigg)^{\!2}\,\Bigg].
			\end{equation}
			Then
			\[
			\max\Big\{\|\mathcal{P}_{k(m+n)}(A)\|^2_{\text{F}},\ \sum_{i=1}^{k}\sigma_i(A)^2\Big\}
			\;\geq\; \frac{k}{\alpha^*}\,\|A\|^2_{\text{F}},
			\]
			that is, at least one of the two exact candidate updates removes a
			$k/\alpha^*$ fraction of the squared norm.
		\end{theorem}
		\begin{proof}
			Write $\alpha=\alpha^*$ and note that $\alpha>\nu/(m+n)$, so
			$\alpha(m+n)>\nu$. Note also that $\rho\geq k\geq1$ implies $L\neq0$ and
			$\nu\geq k(m+n)>0$ implies $S\neq0$. Suppose first that
			\begin{equation}\label{eq:sparse-case}
				\|\mathcal{P}_{\Omega}(A)\|^2_{\text{F}} \;\geq\; \frac{\nu}{\alpha(m+n)}\,\|A\|^2_{\text{F}}.
			\end{equation}
			Since $k(m+n)\leq\nu=\operatorname{card}(\Omega)$, the $k(m+n)$
			largest-magnitude entries of $A$ within $\Omega$ carry at least the
			fraction $k(m+n)/\nu$ of $\|\mathcal{P}_{\Omega}(A)\|^2_{\text{F}}$, and
			hence
			\[
			\|\mathcal{P}_{k(m+n)}(A)\|^2_{\text{F}}
			\;\geq\; \frac{k(m+n)}{\nu}\,\|\mathcal{P}_{\Omega}(A)\|^2_{\text{F}}
			\;\geq\; \frac{k}{\alpha}\,\|A\|^2_{\text{F}},
			\]
			which proves the claim in this case. Next, let $v_1,\ldots,v_k$ be right
			singular vectors of $L$ associated with its $k$ largest singular values,
			and set $\tilde{V}=[v_1~\cdots~v_k]$. If
			\begin{equation}\label{eq:lowrank-case}
				\|A\tilde{V}\|^2_{\text{F}} \;\geq\; \frac{k}{\alpha}\,\|A\|^2_{\text{F}},
			\end{equation}
			then the variational characterization of singular values
			\cite{hogben2013handbook} gives
			\[
			\sum_{i=1}^{k}\sigma_i(A)^2=\max_{V^TV=I_k}\|AV\|^2_{\text{F}}\geq\|A\tilde{V}\|^2_{\text{F}},
			\]
			and the claim again follows.
			
			It remains to show that \eqref{eq:sparse-case} and \eqref{eq:lowrank-case}
			cannot both fail. Suppose they do. Since $L$ has exactly $\rho$ nonzero
			singular values and $\tilde{V}$ spans its $k$ leading right singular
			directions,
			\[
			\|L\tilde{V}\|^2_{\text{F}} \;=\; \sum_{i=1}^{k}\sigma_i(L)^2 \;\geq\; \frac{k}{\rho}\,\|L\|^2_{\text{F}}.
			\]
			Because $A-L=S$ vanishes outside $\Omega$ and $\tilde{V}$ has orthonormal
			columns,
			\[
			\begin{aligned}
				\|\mathcal{P}_{\Omega}(A)\|_{\text{F}}+\|\mathcal{P}_{\Omega}(L)\|_{\text{F}}
				&\;\geq\; \|\mathcal{P}_{\Omega}(A-L)\|_{\text{F}}
				\;=\; \|A-L\|_{\text{F}} \\
				&\;\geq\; \|(A-L)\tilde{V}\|_{\text{F}}
				\;\geq\; \|L\tilde{V}\|_{\text{F}}-\|A\tilde{V}\|_{\text{F}}.
			\end{aligned}
			\]
			Substituting the failures of \eqref{eq:sparse-case} and
			\eqref{eq:lowrank-case} together with
			$\|\mathcal{P}_{\Omega}(L)\|_{\text{F}}=\sqrt{\beta}\,\|L\|_{\text{F}}$
			and $\|L\tilde{V}\|_{\text{F}}\geq\sqrt{k/\rho}\,\|L\|_{\text{F}}$ yields
			\[
			\sqrt{\beta}\,\|L\|_{\text{F}}+\sqrt{\frac{\nu}{\alpha(m+n)}}\,\|A\|_{\text{F}}
			\;>\; \sqrt{\frac{k}{\rho}}\,\|L\|_{\text{F}}-\sqrt{\frac{k}{\alpha}}\,\|A\|_{\text{F}},
			\]
			which rearranges to
			\begin{equation}\label{eq:combined}
				\frac{1}{\sqrt{\alpha}}\Big(\sqrt{k}+\sqrt{\frac{\nu}{m+n}}\Big)\|A\|_{\text{F}}
				\;>\; \Big(\sqrt{\frac{k}{\rho}}-\sqrt{\beta}\Big)\|L\|_{\text{F}}.
			\end{equation}
			By the assumption $\beta<k/\rho$, both sides of \eqref{eq:combined} are
			nonnegative, so the inequality may be squared. Moreover, the failure of
			\eqref{eq:sparse-case}, together with
			$\|A\|^2_{\text{F}}=\|\mathcal{P}_{\Omega}(A)\|^2_{\text{F}}+\|\mathcal{P}_{\Omega^\complement}(A)\|^2_{\text{F}}$
			and the identity
			$\mathcal{P}_{\Omega^\complement}(A)=\mathcal{P}_{\Omega^\complement}(L)$,
			implies
			\[
			\Big(1-\frac{\nu}{\alpha(m+n)}\Big)\|A\|^2_{\text{F}}
			\;<\; \|\mathcal{P}_{\Omega^\complement}(L)\|^2_{\text{F}}
			\;=\; (1-\beta)\,\|L\|^2_{\text{F}}.
			\]
			Combining the square of \eqref{eq:combined} with this bound and dividing
			by $\|L\|^2_{\text{F}}>0$ gives
			\[
			\Big(\sqrt{\frac{k}{\rho}}-\sqrt{\beta}\Big)^{\!2}
			\;<\; \frac{(1-\beta)(m+n)}{\alpha(m+n)-\nu}\Big(\sqrt{k}+\sqrt{\frac{\nu}{m+n}}\Big)^{\!2},
			\]
			which rearranges to
			\[
			\alpha \;<\; \frac{\nu}{m+n}+(1-\beta)\,
			\frac{\big(\sqrt{k}+\sqrt{\nu/(m+n)}\big)^2}{\big(\sqrt{k/\rho}-\sqrt{\beta}\big)^2}
			\;=\; \alpha^*,
			\]
			where the last equality uses
			$\big(\sqrt{k}+\sqrt{\nu/(m+n)}\big)^2=\frac{\nu}{m+n}\big(1+\sqrt{k(m+n)/\nu}\big)^2$.
			This contradicts $\alpha=\alpha^*$ and completes the proof.
		\end{proof}
		
		\begin{corollary}\label{cor:simplified}
			Under the assumptions of Theorem \ref{thm:structured},
			\[
			\alpha^* \;\leq\; \bar{\alpha} \;:=\; \frac{\nu}{m+n}
			\Big(1+\frac{4\rho}{(\sqrt{k}-\sqrt{\beta\rho})^2}\Big),
			\]
			so at least one of the two exact candidate updates removes a
			$k/\bar{\alpha}$ fraction of the squared norm.
		\end{corollary}
		\begin{proof}
			Since $k(m+n)\leq\nu$, we have $\big(1+\sqrt{k(m+n)/\nu}\big)^2\leq4$;
			combining this with $1-\beta\leq1$ and
			$\big(\sqrt{k/\rho}-\sqrt{\beta}\big)^2=(\sqrt{k}-\sqrt{\beta\rho})^2/\rho$
			in \eqref{eq:alpha-star} gives the bound.
		\end{proof}
		
		Theorem \ref{thm:structured} and Corollary \ref{cor:simplified} warrant several remarks, particularly regarding the mildness of their assumptions and the magnitude of the resulting improvement over the worst-case baseline.
		
		Consider the assumptions first. The batch size upper bound, $k \leq \min\{\frac{\nu}{m+n}, \rho\}$, simply ensures that the trust-region budget does not exceed the actual intrinsic rank or normalized sparsity of the latent components. If $\nu<m+n$, the entire sparse component
		already fits within one sparse trust-region update, so excluding this case does
		not represent a practically difficult regime. The substantive assumption is $\beta<k/\rho$,
		which requires that the sparse support not hide too much of the low-rank
		component's energy. Some such incoherence condition is unavoidable, a
		sparse component supported on the dominant entries of $L$ can mask it
		entirely, and analogous conditions underpin the exact-recovery theory of
		sparse-plus-low-rank decompositions~\cite{Candes2009-fm,Chandrasekaran2011-vc}.
		As stated, however, the requirement is weak: if the support of $S$ were
		placed uniformly at random, then $\mathbb{E}[\beta]=\nu/(mn)$, the density
		of the sparse component itself, which for $k=\nu/(m+n)$ lies below the
		threshold $k/\rho$ by a factor of $mn/\big(\rho(m+n)\big)$, for square
		matrices $n/(2\rho)$, which is large precisely when $L$ is genuinely low
		rank. In other words, the assumption excludes only components that are
		adversarially aligned, not typical ones.

		Next, consider the comparison with the baseline. The theoretical guarantee improves upon the $\mathcal{O}(k/p)$ baseline of Proposition \ref{prop:baseline} whenever $\alpha^* < p$, a condition that is overwhelmingly likely to hold in our targeted setting. Because the residual is assumed to be the sum of genuinely sparse and low-rank components, $\nu/(m+n)$ and $\rho$ are significantly smaller than the ambient dimension $p = \min\{m,n\}$. Consequently, $\bar{\alpha}$ (and thus $\alpha^*$) is bounded by these small intrinsic dimensions rather than $p$. This allows the per-iteration reduction $k/\alpha^*$ to completely escape the $\mathcal{O}(1/p)$ dimensional bottleneck. The improvement can be dramatic: if the batch size $k$ is chosen to match the intrinsic structure such that $k \approx \nu/(m+n) \approx \rho$, we have $1+\sqrt{k(m+n)/\nu} \approx 2$ and $\sqrt{k/\rho} \approx 1$, so \eqref{eq:alpha-star} yields $\alpha^* \approx k\big[1+4(1-\beta)/(1-\sqrt{\beta})^2\big] = k(5+3\sqrt{\beta})/(1-\sqrt{\beta})$, where the equality uses $1-\beta=(1-\sqrt{\beta})(1+\sqrt{\beta})$. The guaranteed reduction fraction then simplifies to $k/\alpha^* \approx (1-\sqrt{\beta})/(5+3\sqrt{\beta})$. This is a dimension-independent constant; for highly incoherent components ($\beta \approx 0$), the algorithm removes at least one-fifth of the squared residual norm in a single iteration.

		\begin{remark}[Inexact and randomized projections]\label{rem:randomized}
			The preceding guarantees extend directly to inexact projection subroutines
			whenever their captured energies can be compared with those of the corresponding
			exact projections. Specifically, suppose that, at a given iteration,
			
			\[
			\widehat{\Delta}_S
			\geq
			\tau_S
			\|\mathcal{P}_{k(m+n)}(A_r)\|_{\text{F}}^2,~	\widehat{\Delta}_L
			\geq
			\tau_L
			\sum_{i=1}^{k}\sigma_i(A_r)^2,
			\]
			for some $\tau_S,\tau_L\in(0,1]$, where $\widehat{\Delta}_S$ and
			$\widehat{\Delta}_L$ denote the reductions in squared residual norm produced
			by the inexact sparse and low-rank projections. If the algorithm selects the
			candidate with the larger realized reduction, then its decrease is at least
			
			\[
			\tau
			\max\left\{
			\|\mathcal{P}_{k(m+n)}(A_r)\|_{\text{F}}^2,\,
			\sum_{i=1}^{k}\sigma_i(A_r)^2
			\right\},
			\qquad
			\tau:=\min\{\tau_S,\tau_L\}.
			\]
			
			Consequently, every deterministic contraction bound in this subsection remains
			valid with its guaranteed fractional decrease multiplied by $\tau$. In
			particular, if the above inequalities hold uniformly over the iterations, then
			
			\[
			\|A_r^{(t)}\|_{\text{F}}^2
			\leq
			\left(1-\tau\frac{k}{p}\right)^t
			\|A\|_{\text{F}}^2.
			\]
			
			For randomized projection methods, the values of $\tau_S$ and $\tau_L$, as
			well as the probability with which these inequalities hold, depend on the
			specific sampling scheme, oversampling parameter, number of power iterations,
			and spectral properties of the residual. Standard RSVD error bounds may be
			used to derive a probabilistic value of $\tau_L$ for a specified implementation,
			but they do not imply a universal fixed value of $\tau_L$ without additional
			assumptions. Likewise, a corresponding guarantee for the quantile-based sparse
			projection requires an explicit analysis of the sampling procedure. A
			high-probability convergence guarantee over multiple iterations additionally
			requires controlling the accumulated failure probability across iterations.
		\end{remark}

		\section{Numerical experiments}\label{sec:numerics}
		
		Our numerical study is organized into two parts. We first use real-world data
		from radiotherapy treatment planning to contrast the proposed matrix embedding
		framework with matrix recovery methods, and to compare the exact and
		randomized variants of our trust-region algorithm, R-Trust and R3-Trust. We
		then assess R3-Trust as a matrix recovery method on synthetic data. In both
		parts, we benchmark the proposed algorithms against ADMM (a convex
		optimization baseline), GoDec (a nonconvex optimization baseline), and VB (a
		nonoptimization baseline), reviewed in Section \ref{sec:background}. All code and
		data are publicly available in the \href{https://github.com/Radiotherapy-Optimization/SLME-matrix-embedding}{SLME-matrix-embedding}
		GitHub repository. All experiments were performed in MATLAB R2023a on a
		workstation equipped with an Intel Xeon Gold 5420+ processor and 256\,GB of
		RAM.
		
		\subsection{Matrix embedding versus matrix recovery: real-world data}
		
		In this subsection we consider the dose-influence matrix $A$, the large dense
		operator that constitutes the computational bottleneck in solving the
		large-scale radiotherapy optimization problem \eqref{RO-Problem}. We use the
		data of a lung patient ($308{,}290$ rows, $6{,}752$ columns) and the larger
		data of a prostate patient ($580{,}312$ rows, $6{,}570$ columns), both
		obtained from the open-source PortPy package~\cite{jhanwar2023portpy}. Since
		matrix embedding seeks decompositions with small error and low representation
		cost, we compare all methods on the trade-off between the approximation error
		$\|A-(HW+S)\|_{\text{F}}/\|A\|_{\text{F}}$ and the cost of the decomposition, measured by its
		number of nonzeros $\operatorname{nnz}(H,W,S)$.
		
		\textbf{Embedding is fast, parameter-free, and Pareto-optimal.}
		\Cref{fig:parameter_free} compares the embedding and recovery paradigms on
		the lung case, and three observations stand out. First, matrix embedding is
		parameter-free: our algorithm takes only the matrix $A$ and a target error,
		here $2\%$, and returns the entire Pareto surface of decompositions in a
		single run (the algorithm reduces the error incrementally while adding
		nonzeros, so the surface is generated from left to right). In contrast, each
		recovery method produces a single decomposition for a given choice of its
		hyper-parameters, so its trade-off curve must be traced by sweeping those
		hyper-parameters, which entails repeated solves and careful tuning:
		generating the points in \cref{fig:parameter_free} required $30$ solves for
		GoDec, $12$ for VB, and $20$ for ADMM, although fewer points appear in the
		plots because some solves produced errors and/or nnz values outside the
		plotted axis ranges.\footnote{GoDec was run over a grid of six rank upper
			bounds ($15$ to $40$ in steps of $5$) and five cardinality budgets ($0.5\%$
			to $1.5\%$ of $mn$ in steps of $0.25\%$); VB over twelve values of its
			noise-precision parameter $\beta$ ($5$ to $60$ times a data-driven baseline,
			the inverse mean squared entry of $A$); and ADMM over a grid of five sparse
			penalties $\lambda_S\in\{0.01,0.05,0.1,0.5,1\}$ and four low-rank penalties
			$\lambda_L\in\{10,50,100,500\}$.}
		
		Second, the embedding decompositions are Pareto-optimal relative to the
		recovery baselines: across the full range of budgets, the points generated by
		our algorithm attain a smaller error for the same number of nonzeros,
		equivalently, fewer nonzeros for the same error, than those generated by
		GoDec (top-left), VB (top-right), and ADMM (bottom-left).
		
		Third, matrix embedding is markedly more efficient. Recovering the whole
		trade-off curve requires a single run of the embedding algorithm, whereas the
		recovery methods must be rerun for every hyper-parameter setting. Moreover,
		even one embedding solve is faster than a single recovery solve: as reported
		in the table in \cref{fig:parameter_free}, for the lung data, R3-Trust
		computes the entire surface in about $250$ seconds, whereas GoDec, VB, and
		ADMM require roughly $8{,}400$, $22{,}600$, and $1{,}600$ seconds,
		respectively, to produce a single point. The per-point times are averages
		over the solves in the corresponding sweep, since the runtimes varied
		slightly across hyper-parameter settings; all three baselines were run with
		their default termination criteria. For the larger prostate data, R3-Trust
		computes the entire surface in about $10$ minutes, while all three recovery
		methods ran out of memory on our 256-GB system.
		
		\begin{figure}
			\centering
			\begin{minipage}[b]{0.43\linewidth}
				\centering
				\includegraphics[width=\linewidth]{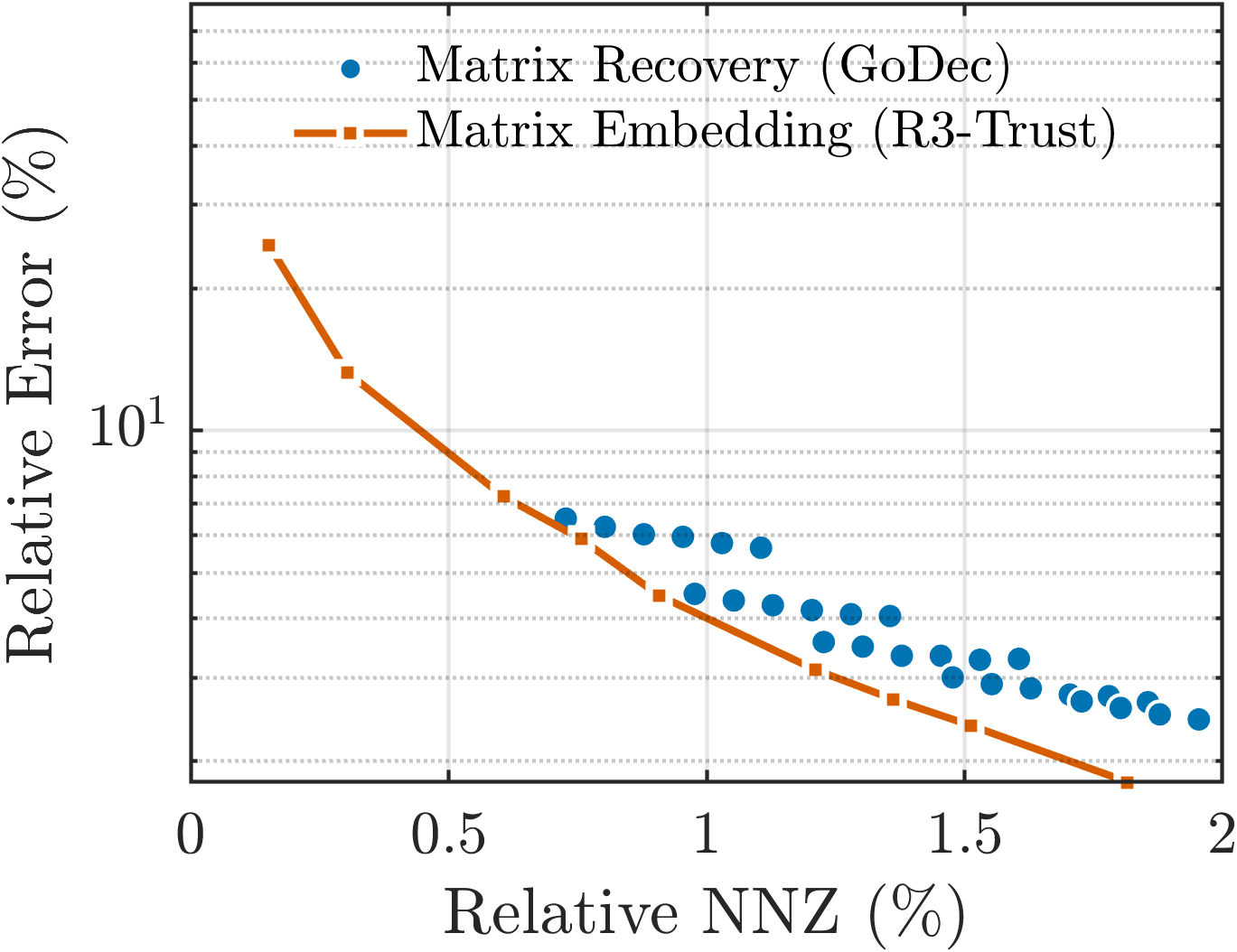}
			\end{minipage}\hfill
			\begin{minipage}[b]{0.43\linewidth}
				\centering
				\includegraphics[width=\linewidth]{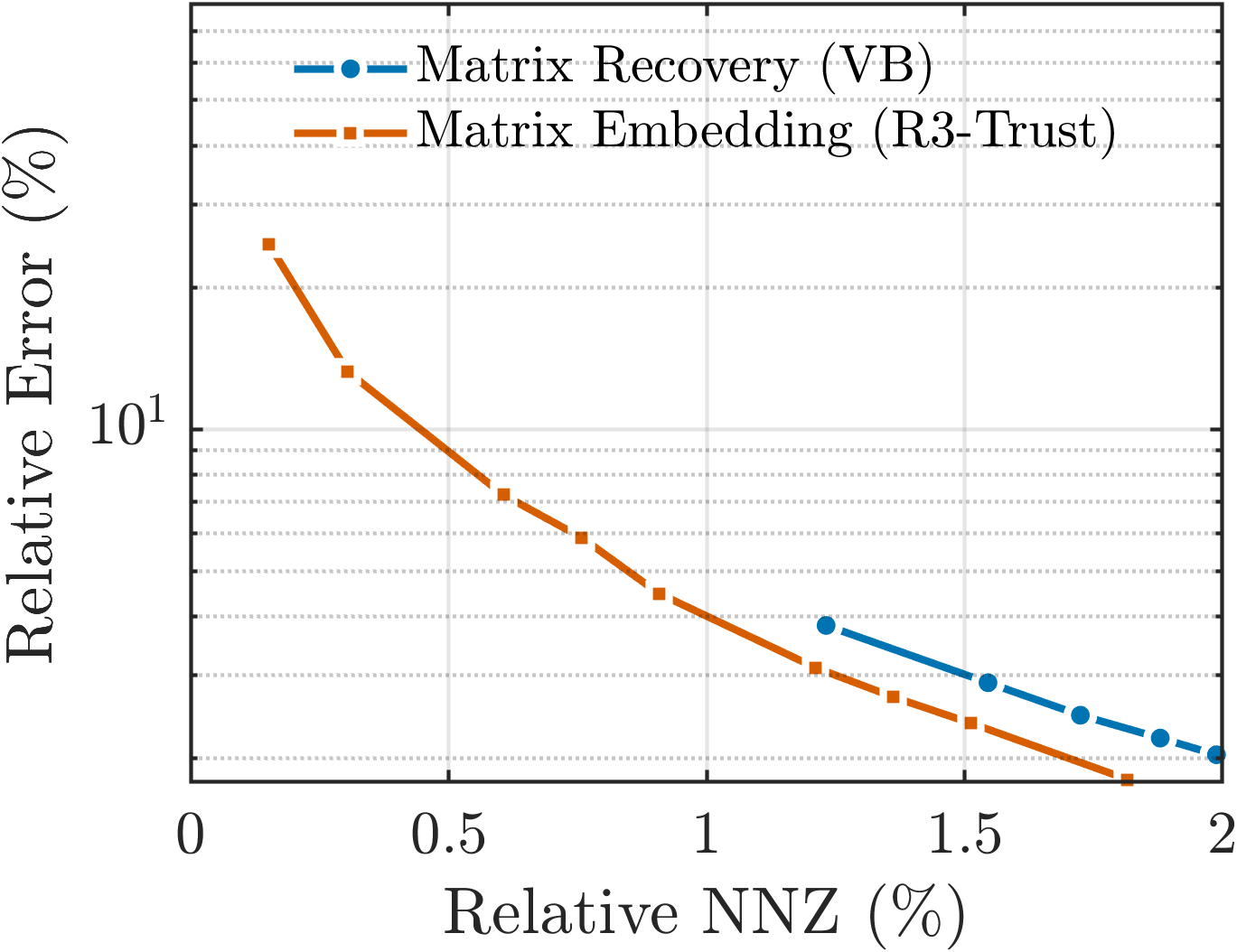}
			\end{minipage}
			
			\vspace{1em}
			
			\begin{minipage}[c]{0.43\linewidth}
				\centering
				\includegraphics[width=\linewidth]{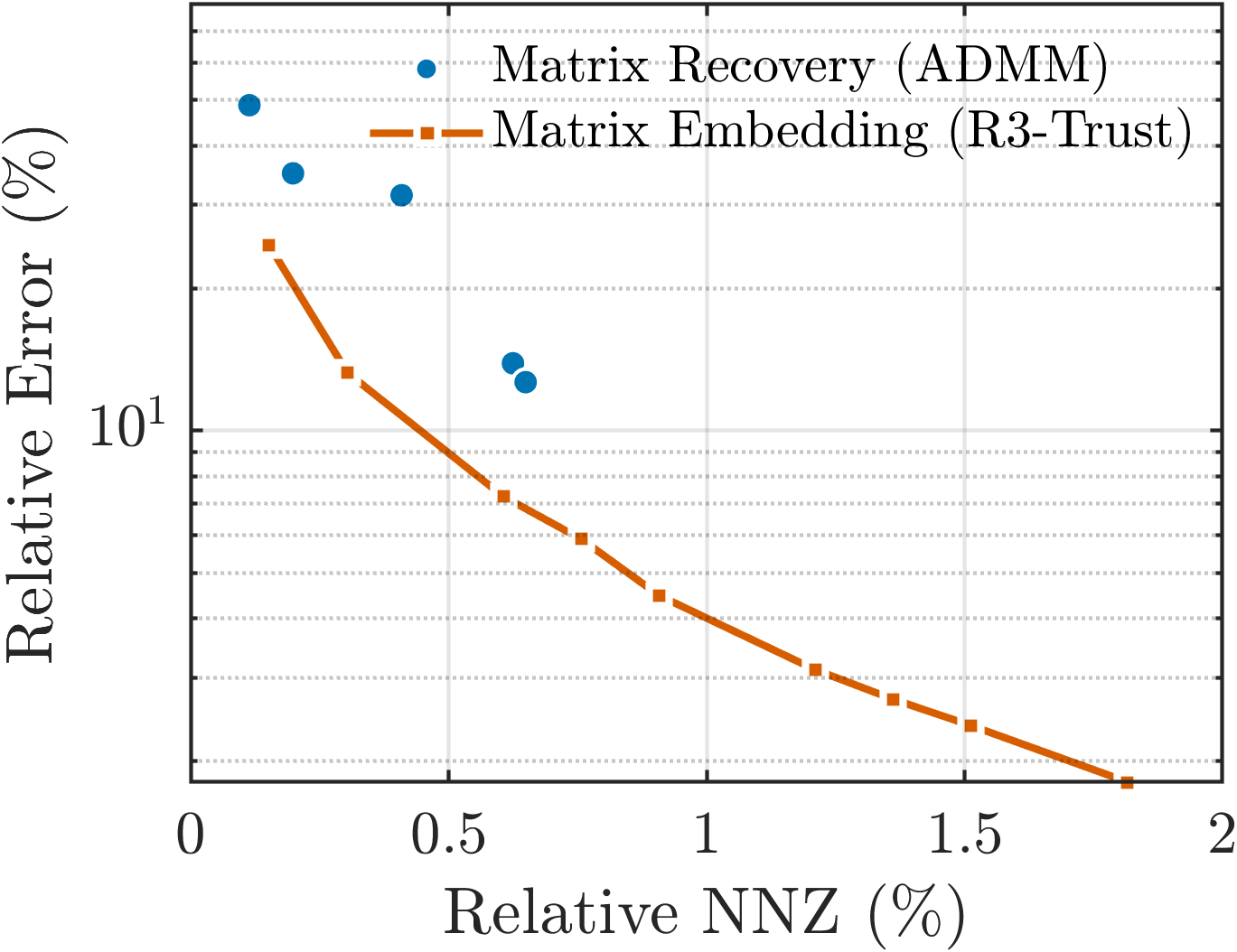}
			\end{minipage}\hfill
			\begin{minipage}[c]{0.48\linewidth}
				\centering
				\small
				\setlength{\tabcolsep}{4pt}
				\begin{tabular}{lcc}
					\hline
					& \multicolumn{2}{c}{Runtime (s)} \\
					\cline{2-3}
					Method                  & Lung   & Prostate \\
					\hline
					R3-Trust (full surface) & $248$  & $608$    \\
					GoDec (per point)       & $8,356$ & OOM      \\
					VB (per point)          & $22,575$ & OOM      \\
					ADMM (per point)        & $1,623$ & OOM      \\
					\hline
				\end{tabular}
			\end{minipage}
			
			\caption{Matrix embedding versus matrix recovery on the lung
				dose-influence matrix. Each panel plots the relative approximation
				error (log scale) against the number of nonzeros of the decomposition
				relative to $mn$. In a single parameter-free run, R3-Trust generates
				the entire error--nnz Pareto surface (solid curve), which dominates
				the individual decompositions produced by sweeping the
				hyper-parameters of GoDec (top-left), VB (top-right), and ADMM
				(bottom-left). The table (bottom-right) reports wall-clock runtimes
				for the lung and prostate cases: one full-surface run of R3-Trust
				versus the average per-point solve time of each recovery method. On
				the prostate case, all three recovery methods ran out of memory
				(OOM).}
			\label{fig:parameter_free}
		\end{figure}

		\textbf{R-Trust versus R3-Trust.} \Cref{fig:rtrust_vs_r3trust} compares the exact
		R-Trust algorithm with its randomized counterpart R3-Trust on the same lung
		data, in terms of quality (left panel) and speed (right panel). In principle,
		the randomization and the batch size larger than one ($10$ in our
		experiments) introduced by R3-Trust could degrade the quality of the
		decomposition, while its rectification step could improve it.
		\Cref{fig:rtrust_vs_r3trust} (left) shows that the net effect is minimal: the
		two algorithms trace nearly identical Pareto surfaces in the error--nnz
		plane. Computationally, however, R3-Trust is roughly an order of magnitude
		faster than R-Trust (right panel), since it avoids the exact rank-one SVD and
		the full sort of the residual entries, and expands the trust region by using
		a larger batch at each iteration. The prostate case exhibits the same trend
		(\cref{fig:rtrust_vs_r3trust_prostate} in the appendix).
		
		\begin{figure}
			\centering
			\includegraphics[width=.42\linewidth]{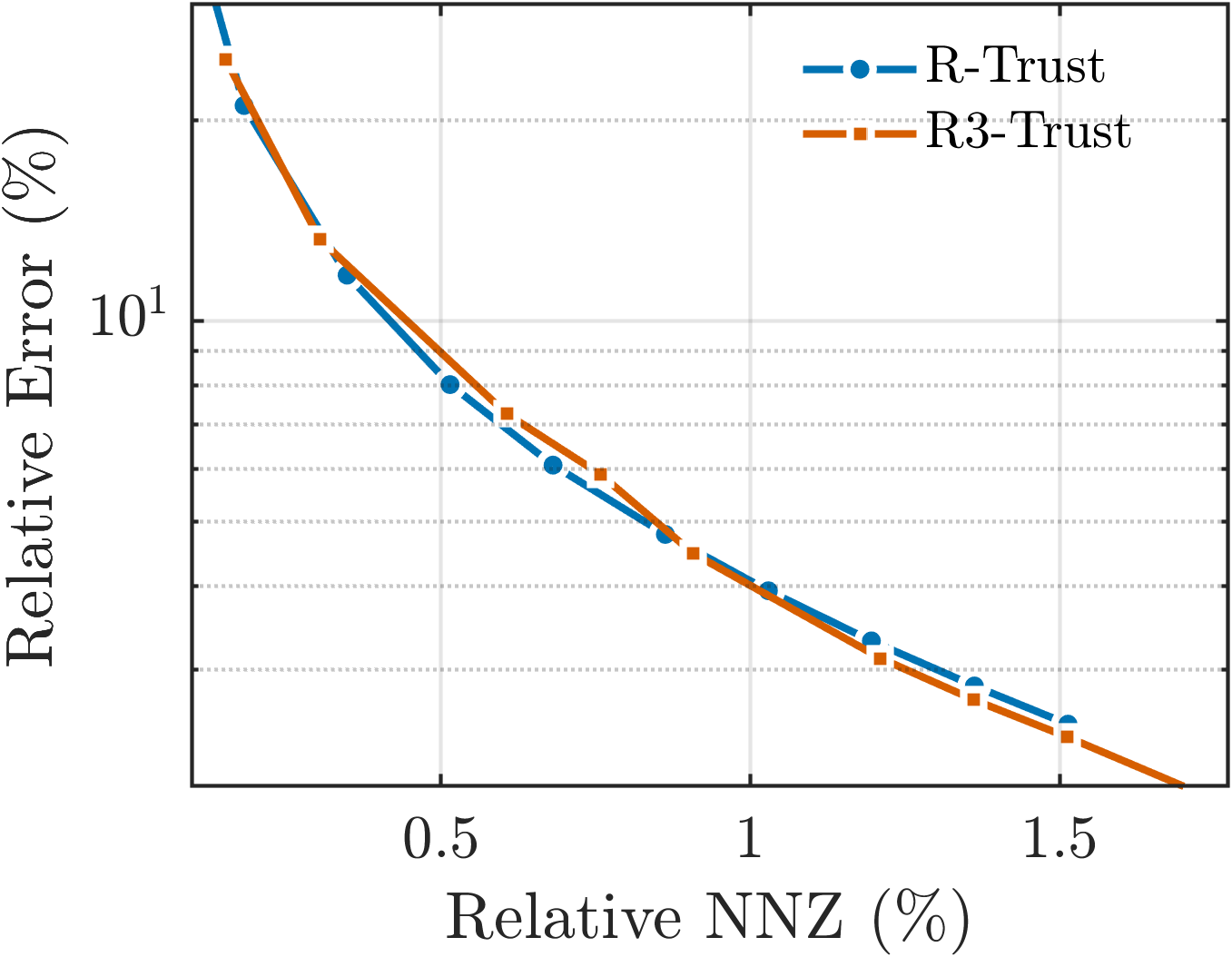}
			\includegraphics[width=.43\linewidth]{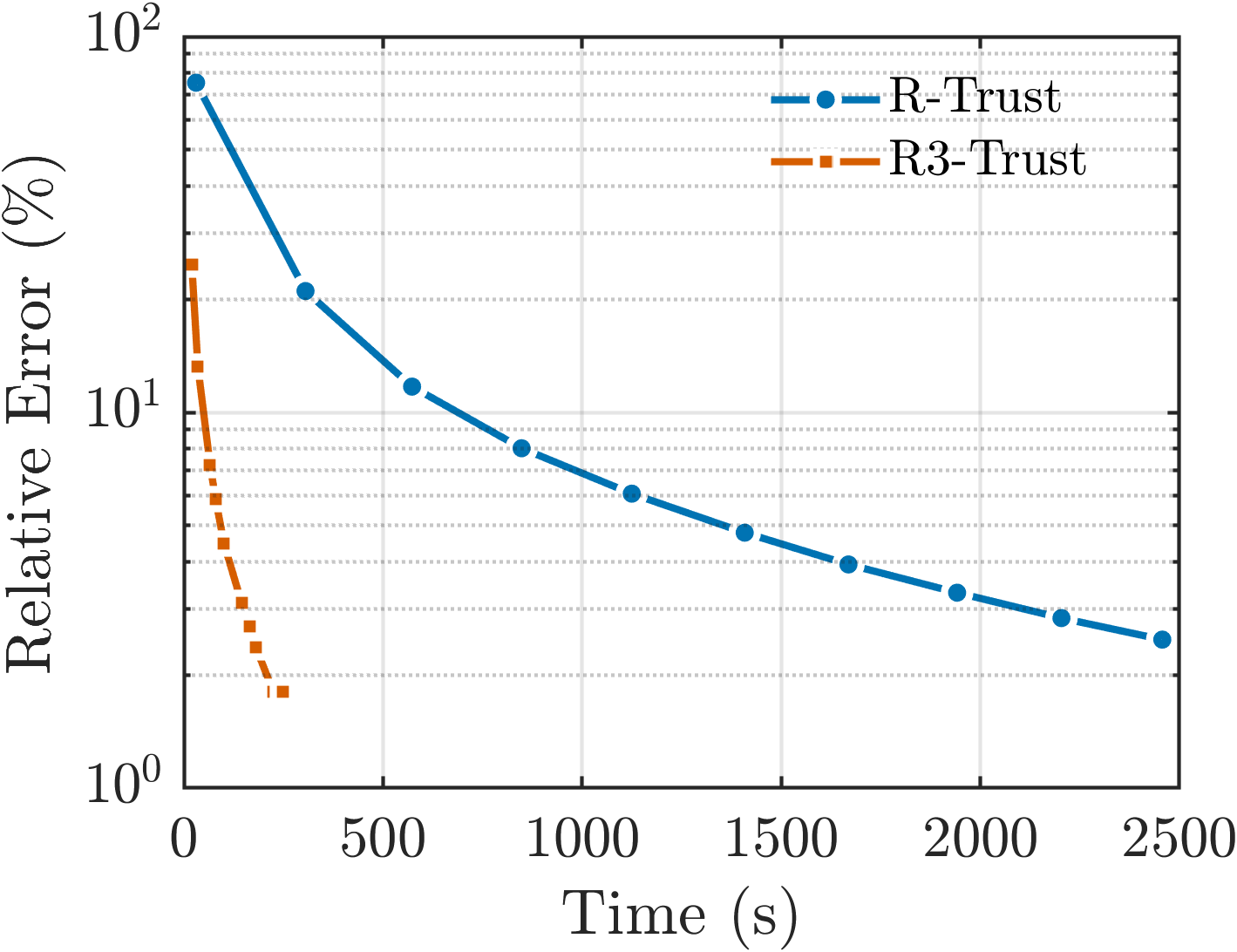}
			\caption{Exact (R-Trust) versus randomized (R3-Trust) trust-region
				embedding on the lung case. Left: the two algorithms trace nearly
				identical error--nnz Pareto surfaces. Right: relative error versus
				runtime; R3-Trust (batch size $10$) reaches the same accuracy roughly
				an order of magnitude faster than R-Trust.}
			\label{fig:rtrust_vs_r3trust}
		\end{figure}
		
		Finally, we examine the effect of the batch size in R3-Trust, that is, the
		rank of the low-rank term appended at each iteration.
		\Cref{fig:r3trust_parameters_sensitivity_analysis} shows that increasing the
		batch size accelerates the algorithm (right), because more of the low-rank
		structure is captured per iteration and fewer iterations are needed, but that
		an excessively large batch size can slightly worsen the error--nnz trade-off
		(left) by appending low-rank terms that are less economical than a finer,
		alternating selection. From the trust-region point of view, a larger batch
		size expands the trust region at each iteration, trading some accuracy per
		stored value for larger steps and faster progress. A batch size of $10$
		offers a favorable balance between speed and quality and is adopted as the
		default in our implementation and in all experiments in this section (see
		\cref{fig:r3trust_parameters_sensitivity_analysis_prostate} in the appendix
		for the prostate case, which shows the same trend).
		
		One may ask whether the batch size contradicts our claim that embedding is
		parameter-free. It does not, because the batch size plays a fundamentally
		different role than the hyper-parameters of the recovery methods. Like a
		stopping tolerance or an iteration limit, parameters that every numerical
		method carries, it balances accuracy against computational effort, and
		regardless of its value a single run still returns the entire error--nnz
		Pareto surface. The recovery hyper-parameters (rank and cardinality bounds,
		regularization penalties, noise precisions), by contrast, determine
		\emph{which} single decomposition is returned: they must be tuned anew for
		every desired point on the trade-off curve, or, in the recovery paradigm, to
		match the unknown underlying structure.
		
		\begin{figure}
			\centering
			\includegraphics[width=.42\linewidth]{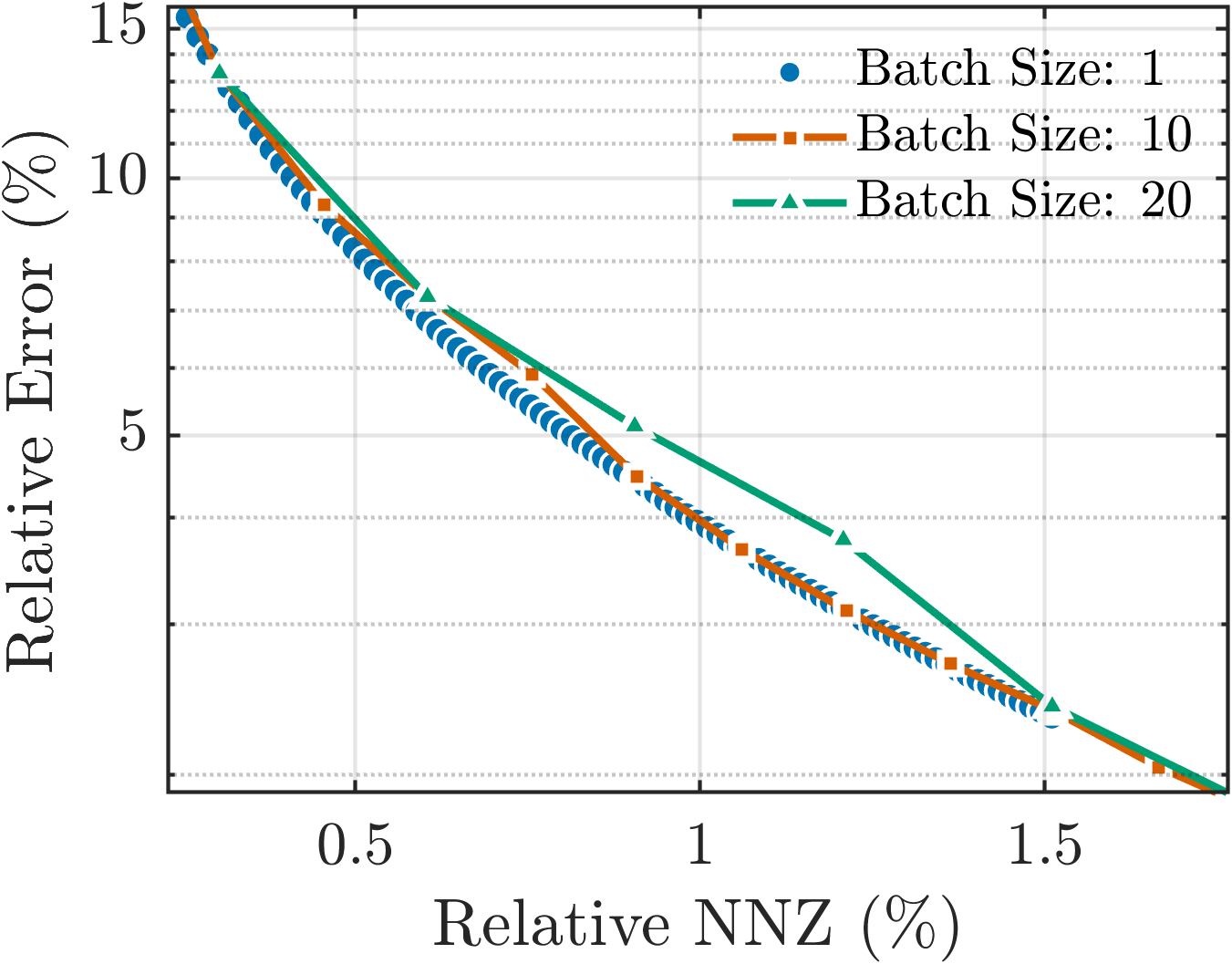}
			\includegraphics[width=.43\linewidth]{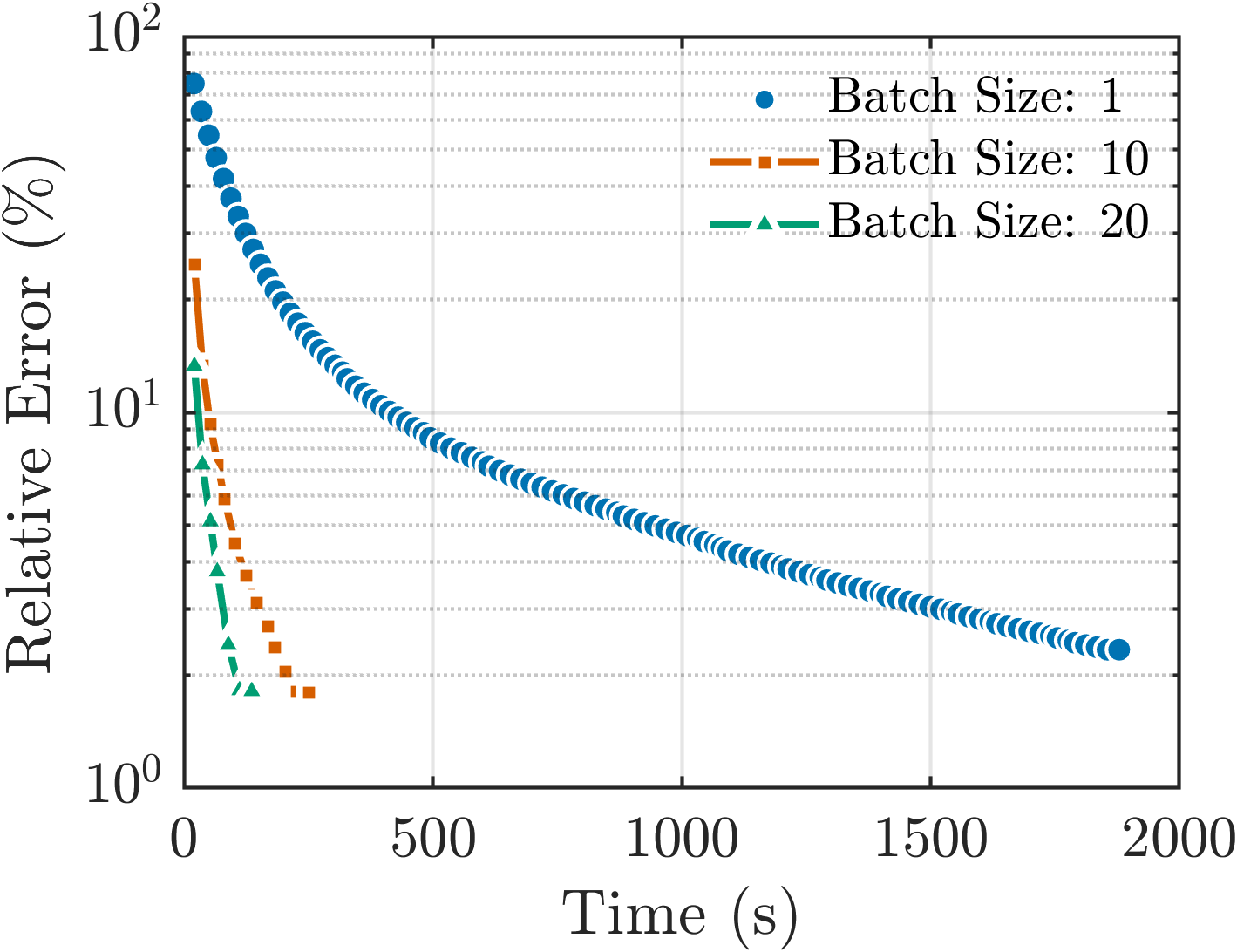}
			\caption{Sensitivity of R3-Trust to the batch size (the rank of the
				low-rank term added per iteration) on the lung case, for batch sizes
				$1$, $10$, and $20$. Increasing the batch size reduces runtime
				(right) but can slightly degrade the error--nnz trade-off (left). A
				batch size of $10$ is used as the default.}
			\label{fig:r3trust_parameters_sensitivity_analysis}
		\end{figure}

		\subsection{Using R3-Trust for matrix recovery: synthetic data}
		We now evaluate R3-Trust as a matrix \emph{recovery} method, following the
		synthetic-data protocol of the comparative study of Zhou et
		al.~\cite{Zhou2014-ug}. The ground-truth low-rank component is generated as
		the product of two Gaussian factors, $L=UV$ with
		$U\in\mathbb{R}^{m\times r}$ and $V\in\mathbb{R}^{r\times n}$ having
		independent standard normal entries, normalized so that
		$\|L\|_{\text{F}}=\sqrt{mn}$. The sparse component $S$ has support drawn uniformly at
		random with density $\rho$ and nonzero entries drawn uniformly from
		$[-10,10]$, and the noise matrix $E$ has independent $\mathcal{N}(0,\sigma^2)$
		entries. The observed matrix is $A=L+S+E$ with $m=n=1000$ and $r=50$, and,
		again following~\cite{Zhou2014-ug}, we consider four regimes obtained by
		crossing the sparsity levels $\rho\in\{0.1,0.3\}$ with the noise levels
		$\sigma\in\{0,0.1\}$. (With a slight abuse of notation, $\rho$ and
		$\sigma$ in this section denote the sparsity density and the noise level,
		and are unrelated to the rank parameter $\rho$ and the singular values
		$\sigma_i$ of Section~\ref{sec:algorithms}.) Recovery quality is measured against the ground truth
		by the mean of the relative errors of the two components,
		$\tfrac12\big(\|L-\hat L\|_{\text{F}}/\|L\|_{\text{F}}+\|S-\hat S\|_{\text{F}}/\|S\|_{\text{F}}\big)$.
		
		The hyper-parameter choices are summarized in \cref{tab:hyperparams}.
		R3-Trust and VB run with their default settings, identical to those used in
		the radiotherapy experiments (in the noiseless regime, R3-Trust is simply run
		to a tighter stopping tolerance, since exact recovery is possible). GoDec, in
		contrast, requires upper bounds on the rank and cardinality, which we set
		from the ground truth as $1.5\,r$ and $1.5\operatorname{nnz}(S)$; ADMM in the
		noisy regime likewise requires penalties set from the known noise level.
		These choices give GoDec and ADMM access to information about the underlying
		structure that would be unavailable in practice.
		
		\Cref{fig:matrix_recovery_syntetic} reports, for each regime, the recovery error as a function of runtime, with each curve traced over the iterations of the corresponding algorithm. Across the four regimes, R3-Trust delivers competitive and often superior performance, most notably in speed. Thus, although R3-Trust was designed for embedding, these results suggest that it is also an attractive option for matrix recovery, particularly when runtime is a concern.
		
We emphasize that, although R3-Trust was run here with its default settings, we
do not claim it to be parameter-free for \emph{recovery} applications. The
embedding formulation~\eqref{LSME} seeks a computationally efficient
representation of the given matrix, which need not coincide with the
ground-truth decomposition; consequently, the default settings do not always
recover the underlying components well. Two adjustments help align the algorithm
with the structure being recovered. First, the batch size 
$k$ should be reduced when a single default-sized update would add more structure than the ground truth actually contains: if the rank of the underlying low-rank component is at
or below the default batch size ($k=10$), a smaller $k$ prevents the algorithm
from appending more rank per iteration than the ground truth warrants, and the
same applies to a highly sparse component, whose nonzero count $\|S\|_0$ is
comparable to or smaller than the per-iteration sparse budget $k(m+n)$. Second,
a cost weight $\eta$ can be introduced into the objective of~\eqref{LSME},
replacing $(m+n)\operatorname{rank}(L)+\|S\|_0$ with
$\eta(m+n)\operatorname{rank}(L)+\|S\|_0$. Such a weight is unnecessary for
embedding, where the relative cost of the low-rank and sparse parts is fixed by
the computational framework, but in recovery it becomes a useful tunable control:
increasing $\eta$ makes low-rank updates more expensive relative to sparse ones
and drives the decomposition toward a lower-rank output, whereas decreasing it
favors the low-rank component. Thus $\eta$ can be tuned according to whether the
ground truth is predominantly low-rank or sparse. Introducing this weight
requires only minor changes to the algorithms and is already supported in the
\href{https://github.com/Radiotherapy-Optimization/SLME-matrix-embedding}{released code}.

		\begin{table}[htbp]
			\centering
			\caption{Hyper-parameter selection for the synthetic low-rank plus sparse
				recovery experiments. R3-Trust, VB, and ADMM in the noiseless regime
				($\sigma=0$) use only the algorithms' default settings, identical to those
				used for the radiotherapy data, whereas GoDec and ADMM with noise
				($\sigma>0$) require problem-specific tuning.}
			\label{tab:hyperparams}
			\begin{tabularx}{\linewidth}{@{}l l X@{}}
				\toprule
				Algorithm & Hyper-parameter setting & Notes \\
				\midrule
				R3-Trust
				& Defaults
				& Same as the radiotherapy experiments. \\[3pt]
				VB
				& Defaults
				& Same as the radiotherapy experiments. \\[3pt]
				
				GoDec
				& \begin{tabular}[t]{@{}l@{}}rank $\le \lceil 1.5\,r\rceil$\\[3pt]card $\le \lceil 1.5\,\mathrm{nnz}(S)\rceil$\end{tabular}
				& Problem-specific rank and cardinality upper bounds set from the ground-truth $L$ and $S$. Tighter upper-bounds could improve the performance, but they are unknown in practice. \\[10pt]
				ADMM ($\sigma=0$)
				& $\lambda = 1/\sqrt{\max(m,n)}$
				& Standard PCP penalty, no problem-specific tuning. \\[3pt]
				
				ADMM ($\sigma>0$)
				& \begin{tabular}[t]{@{}l@{}}$\lambda = 1/\sqrt{\max(m,n)}$\\[3pt]$\mu = 1/(\sqrt{2\max(m,n)}\,\sigma)$\end{tabular}
				& Problem-specific penalties based on the noise level $\sigma$. \\[3pt]
				\bottomrule
			\end{tabularx}
		\end{table}

		\begin{figure}
			\centering
			\begin{minipage}[t]{0.48\linewidth}
				\centering
				\includegraphics[width=\linewidth]{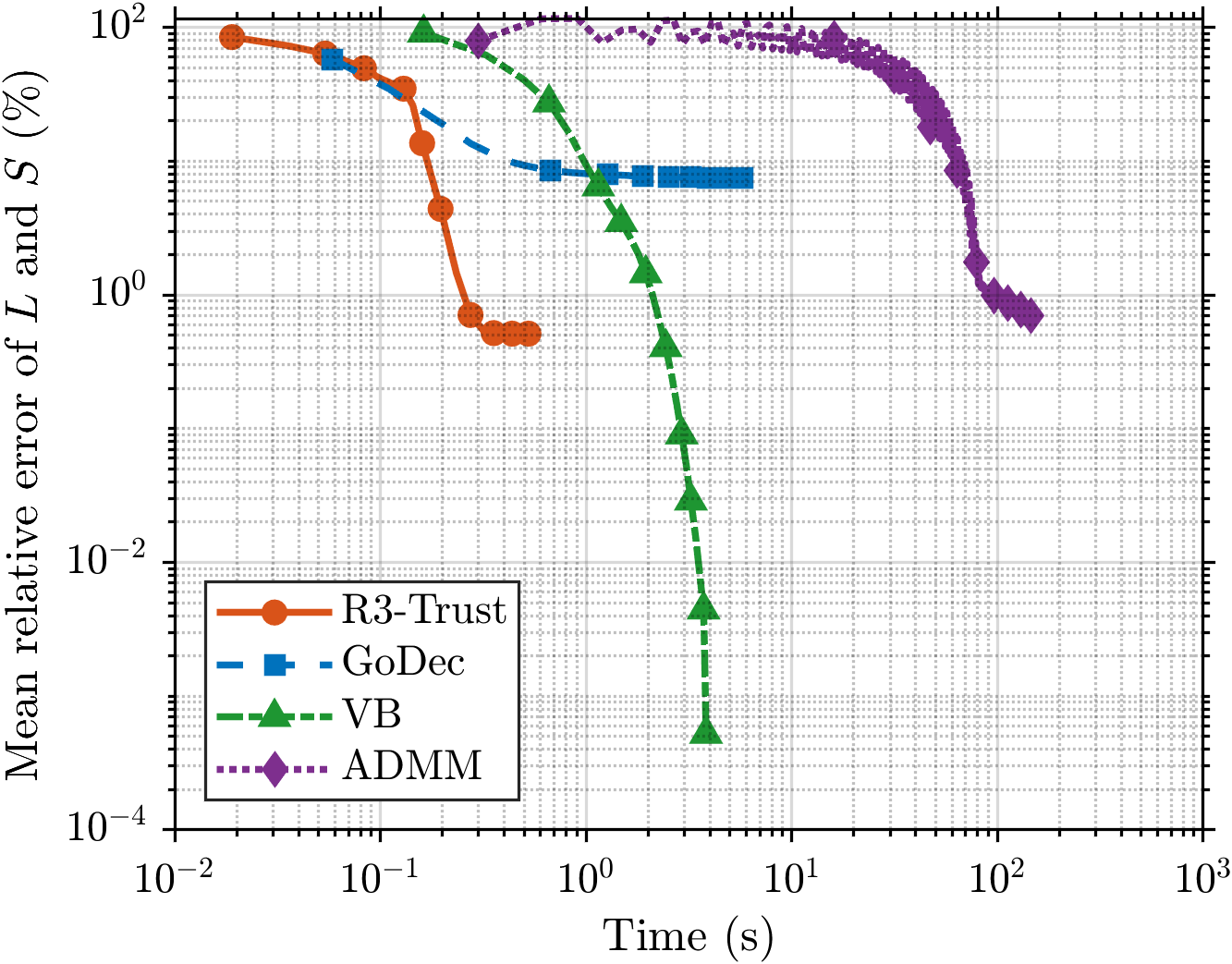}\\
				{\small (a) $\rho=0.1$, $\sigma=0$}
			\end{minipage}\hfill
			\begin{minipage}[t]{0.48\linewidth}
				\centering
				\includegraphics[width=\linewidth]{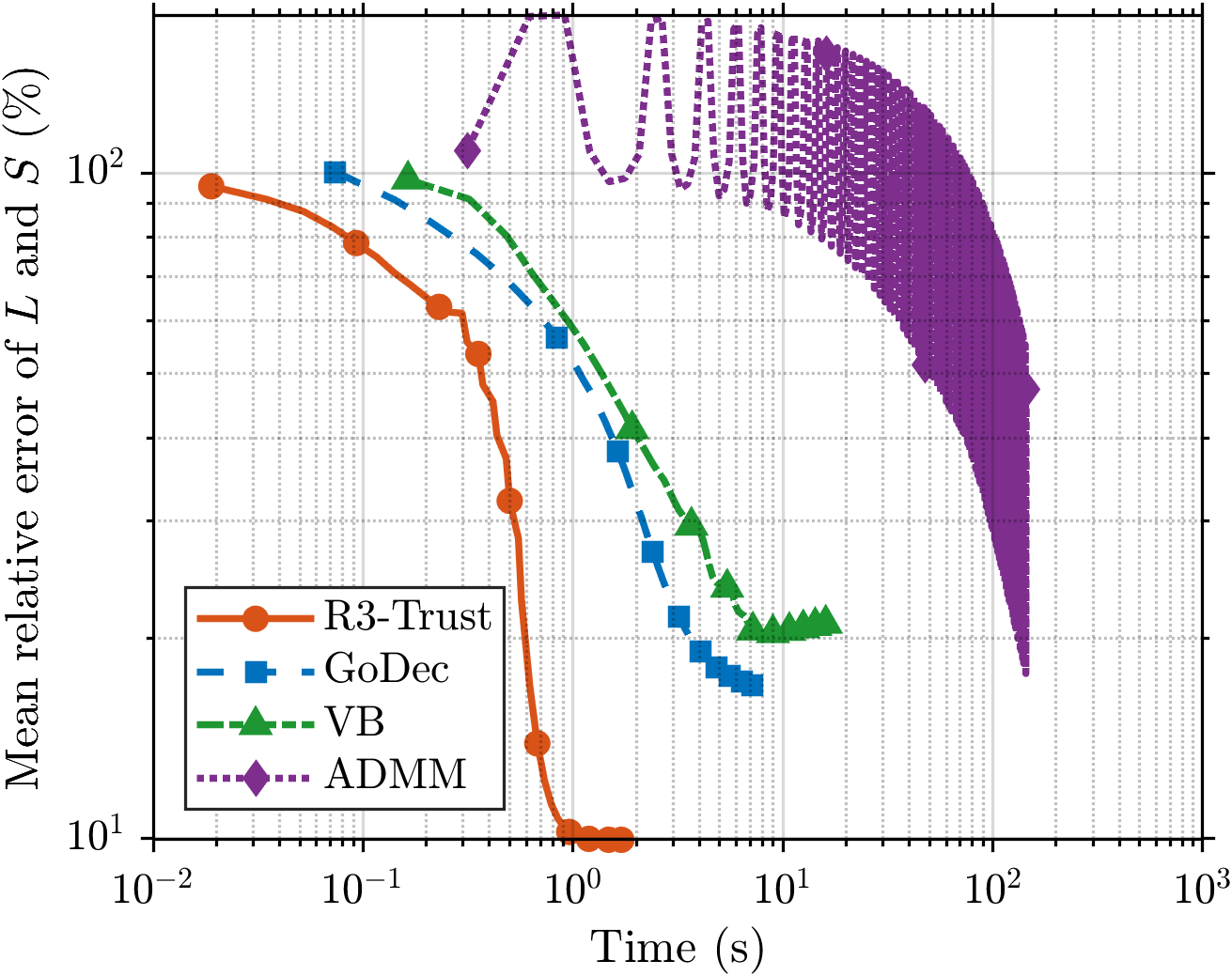}\\
				{\small (b) $\rho=0.3$, $\sigma=0$}
			\end{minipage}\\[8pt]
			\begin{minipage}[t]{0.48\linewidth}
				\centering
				\includegraphics[width=\linewidth]{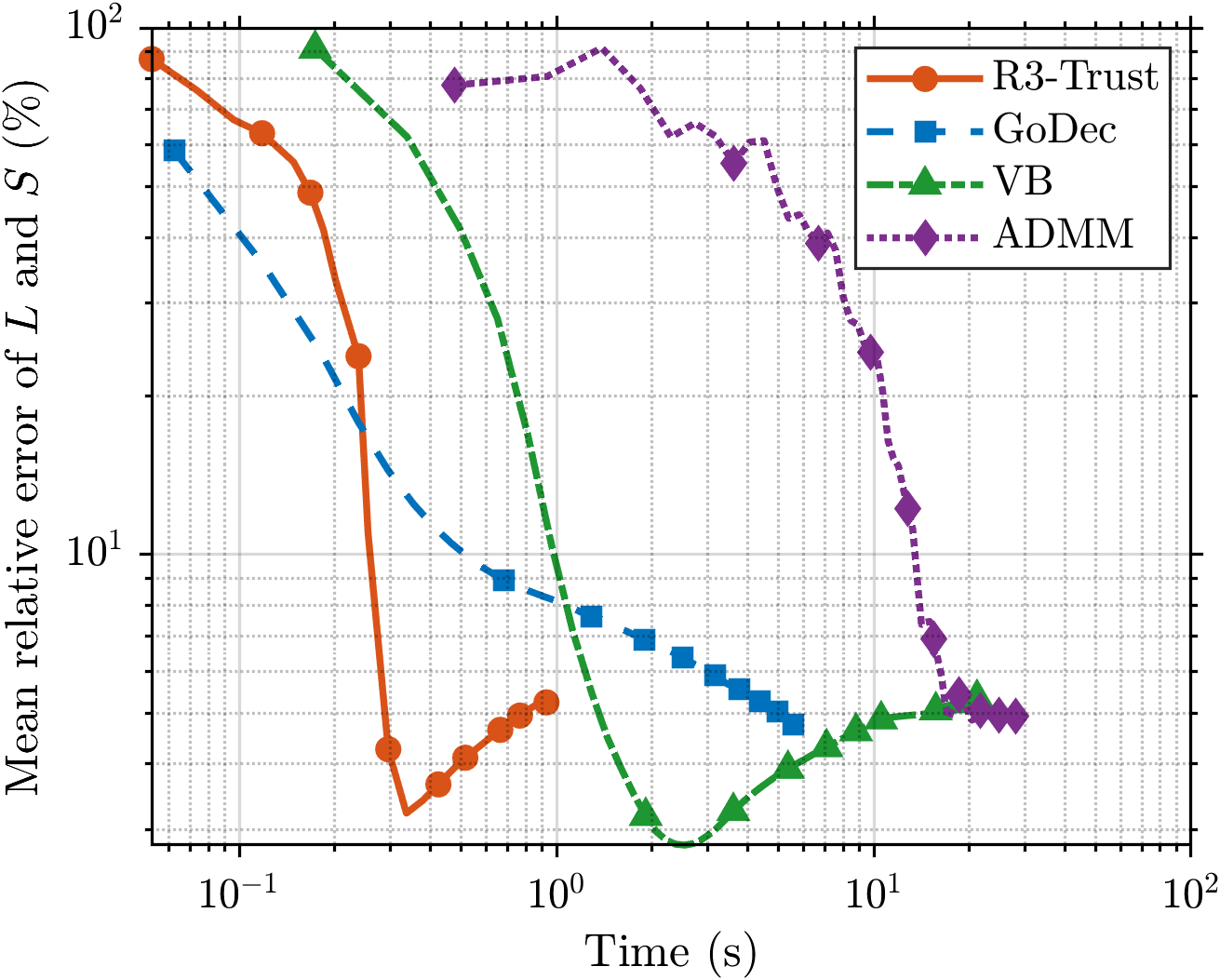}\\
				{\small (c) $\rho=0.1$, $\sigma=0.1$}
			\end{minipage}\hfill
			\begin{minipage}[t]{0.48\linewidth}
				\centering
				\includegraphics[width=\linewidth]{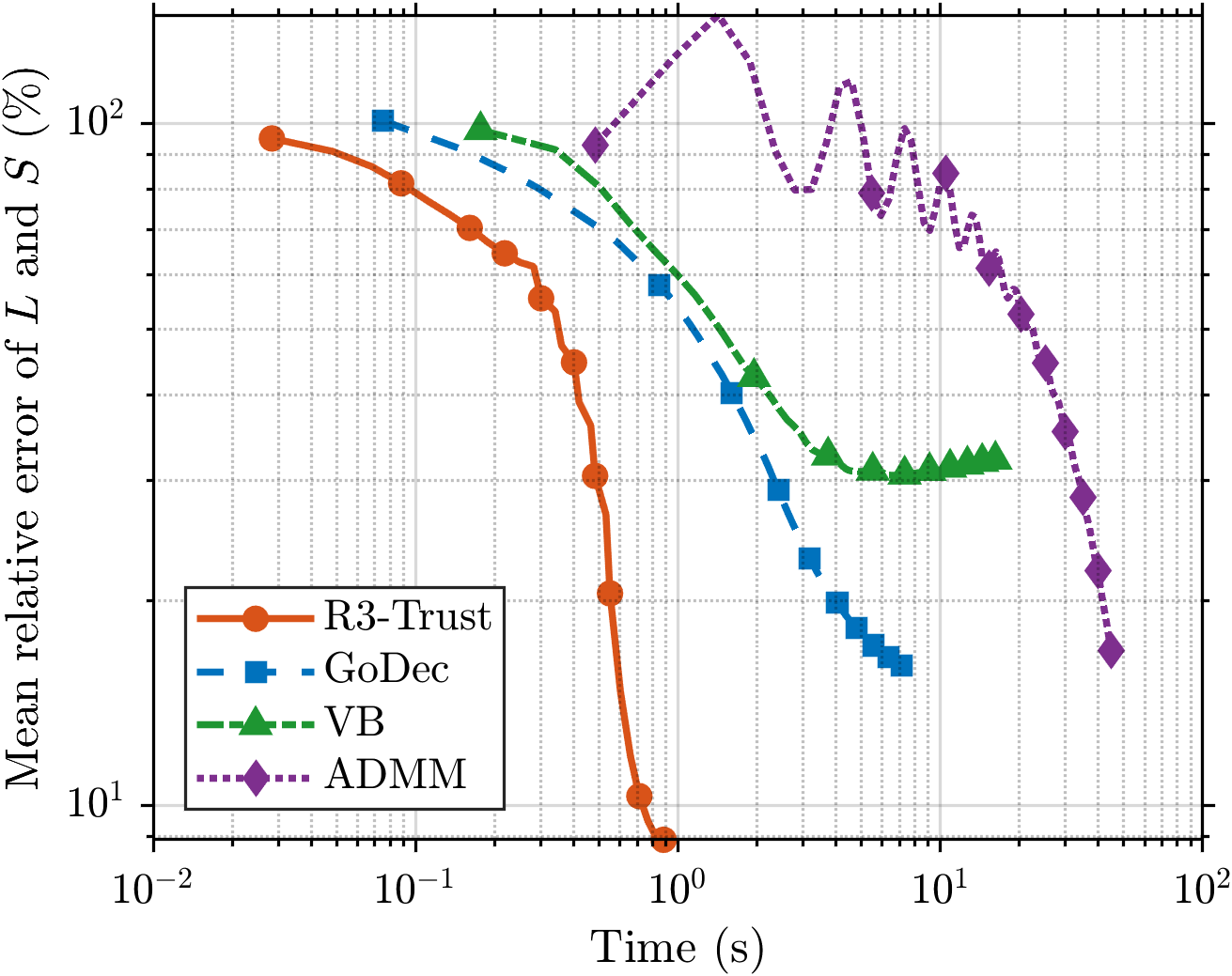}\\
				{\small (d) $\rho=0.3$, $\sigma=0.1$}
			\end{minipage}
			\caption{Matrix recovery on synthetic data with $m=n=1000$ and $r=50$.
				Each panel plots, on log--log axes, the mean of the relative recovery
				errors of $L$ and $S$ (in percent) against runtime, with each curve
				traced over the iterations of the corresponding algorithm. The four
				panels cross the sparsity level $\rho\in\{0.1,0.3\}$ of the sparse
				component (increasing left to right) with the noise level
				$\sigma\in\{0,0.1\}$ (noiseless in the top row, noisy in the bottom
				row). R3-Trust reaches its final accuracy one to two orders of
				magnitude faster than GoDec, VB, and ADMM in all four regimes, and
				attains the best final accuracy in the low-sparsity regimes (b) and
				(d).}
			\label{fig:matrix_recovery_syntetic}
		\end{figure}

		\section{Conclusion}\label{sec:conclusion}
		
		We introduced sparse-plus-low-rank matrix embedding, a new perspective
		on sparse-plus-low-rank decomposition in which the goal is not to recover
		latent components but to construct a computationally efficient surrogate
		$A\approx S+HW$ for a large dense matrix. We formulated SLME as a
		bi-objective nonconvex problem that trades approximation error against
		representation cost, and developed a family of trust-region algorithms
		(R-Trust, R2-Trust, and R3-Trust) that approximate the entire Pareto frontier
		of this trade-off in a single parameter-free run. The algorithms rest on a
		closed-form solution of the trust-region subproblem and come with convergence
		guarantees: unconditional linear convergence, a quadratically improved rate
		for nonnegative matrices, and a per-iteration reduction that becomes
		dimension-independent for matrices with genuine sparse-plus-low-rank
		structure. On clinical dose-influence matrices from radiotherapy treatment
		planning, R3-Trust computed the full accuracy--cost frontier in minutes,
		dominating the decompositions produced by representative recovery
		methods, each of which required hours per point or exhausted memory, and on
		synthetic data it proved competitive with these methods for matrix recovery
		itself.
		
		An exciting direction for future work arises from recent studies in machine
		learning that employ sparse-plus-low-rank decompositions to compress large
		neural networks, for instance by approximating weight matrices of large
		language models~\cite{li2023losparse} or attention
		matrices~\cite{chen2021scatterbrain}. These works share our motivation of
		using the decomposition for compression rather than for recovering latent
		structure. Their setting is more challenging, however: the compression must
		preserve the downstream performance of the network, whereas in our
		application the Frobenius-norm approximation error is the quantity of
		interest. Extending the SLME framework and the R3-Trust algorithm to such
		performance-aware compression objectives is a natural and challenging avenue
		for future research.

		\appendix
		\section{Additional results for the prostate case}\label{sec:appendix}
		
		This appendix collects the prostate-case counterparts of the lung-case
		experiments of Section \ref{sec:numerics}; the same trends are observed.
		
		\begin{figure}
			\centering
			\includegraphics[width=.42\linewidth]{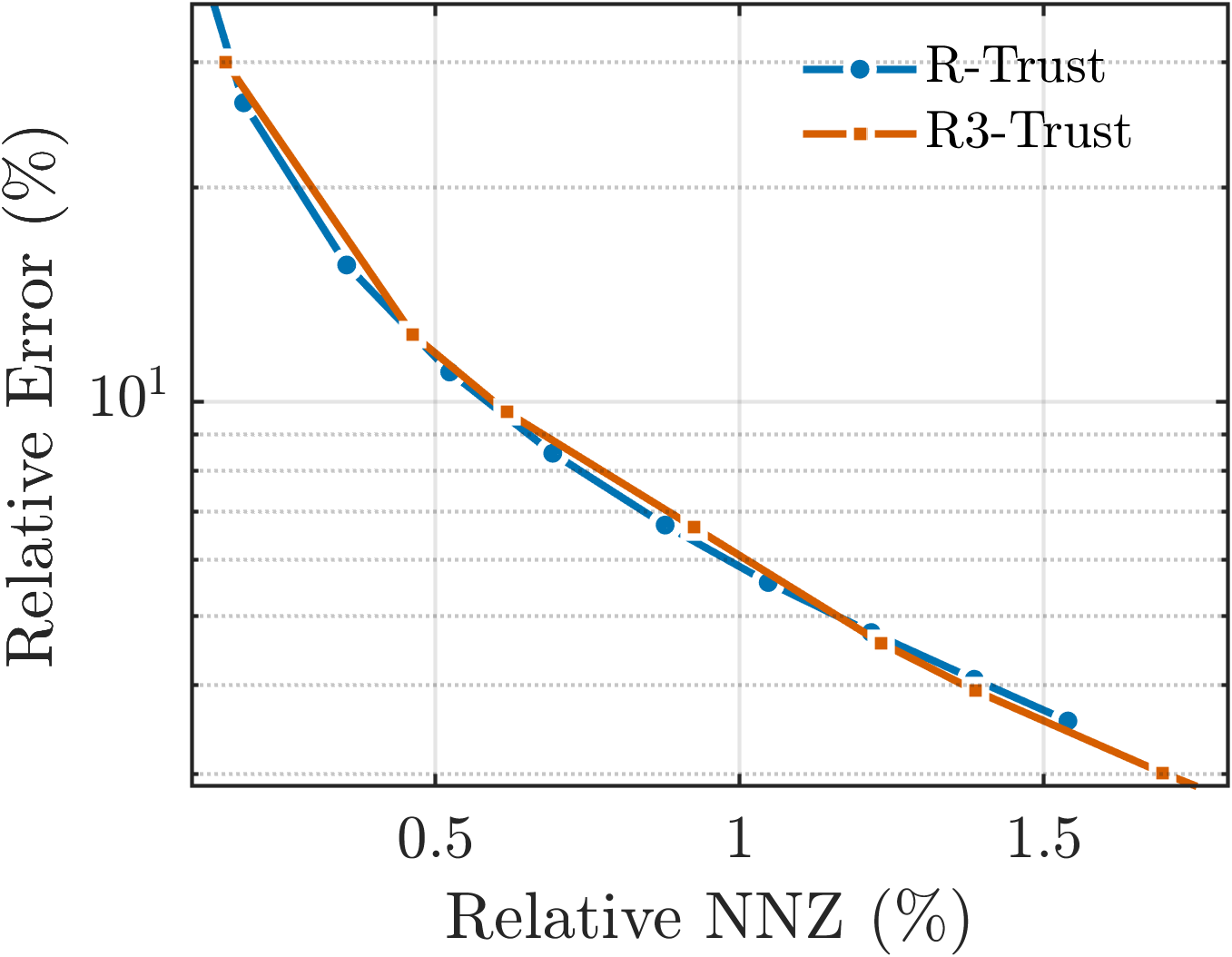}
			\includegraphics[width=.43\linewidth]{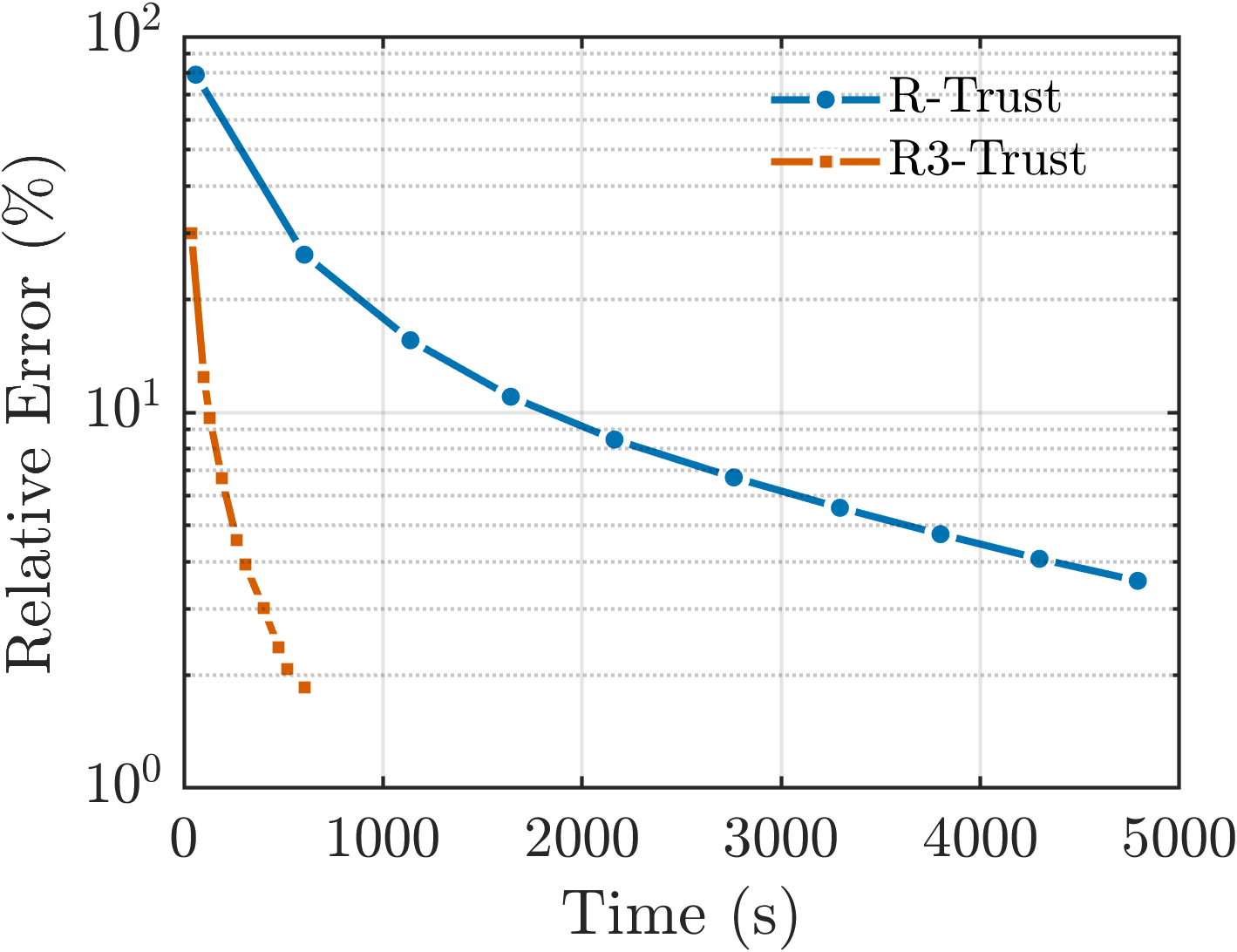}
			\caption{Exact (R-Trust) versus randomized (R3-Trust) trust-region
				embedding on the prostate case. Left: the two algorithms trace nearly
				identical error--nnz Pareto surfaces. Right: relative error versus
				runtime; R3-Trust (batch size $10$) reaches the same accuracy roughly
				an order of magnitude faster than R-Trust.}
			\label{fig:rtrust_vs_r3trust_prostate}
		\end{figure}

		\begin{figure}
			\centering
			\includegraphics[width=.42\linewidth]{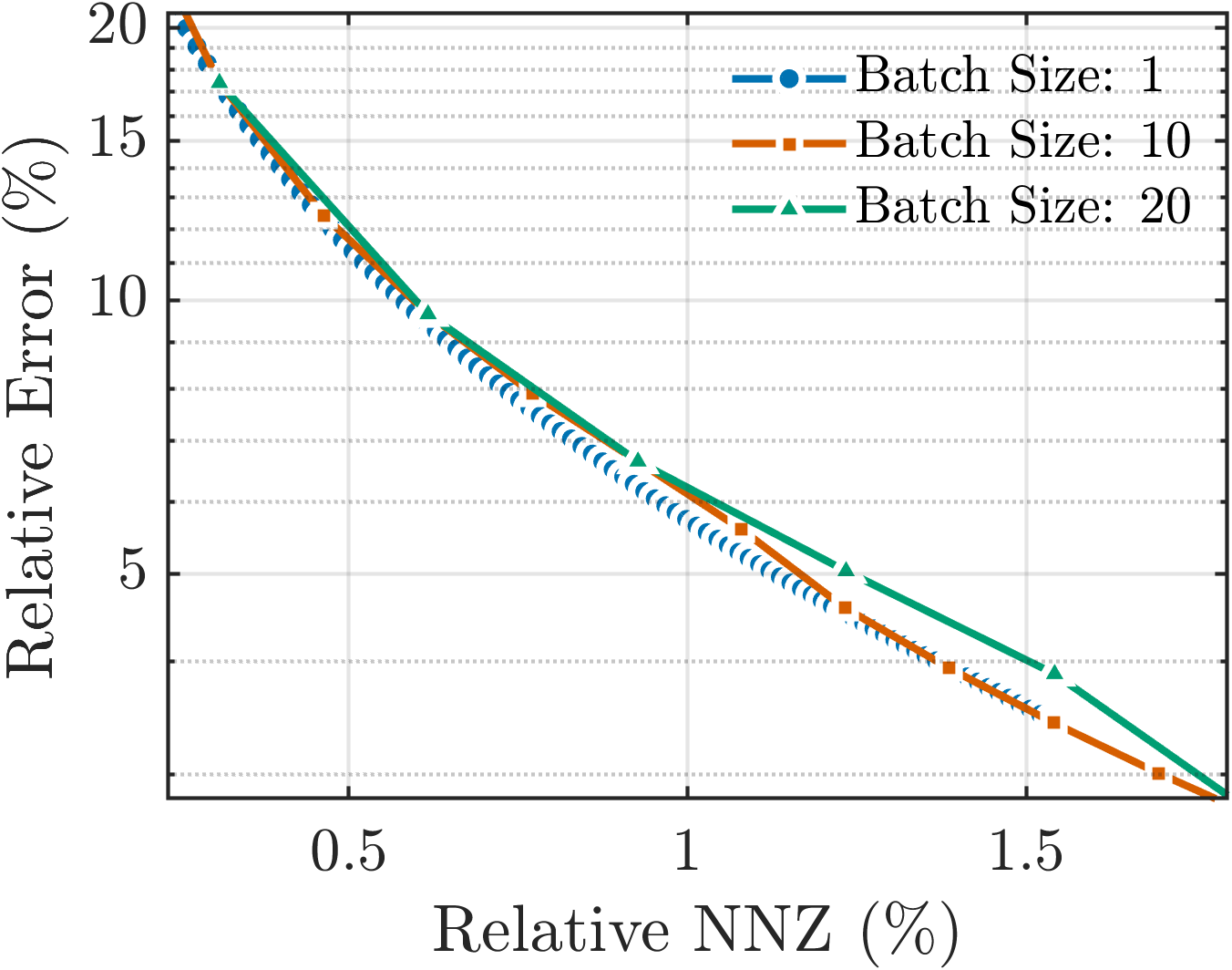}
			\includegraphics[width=.43\linewidth]{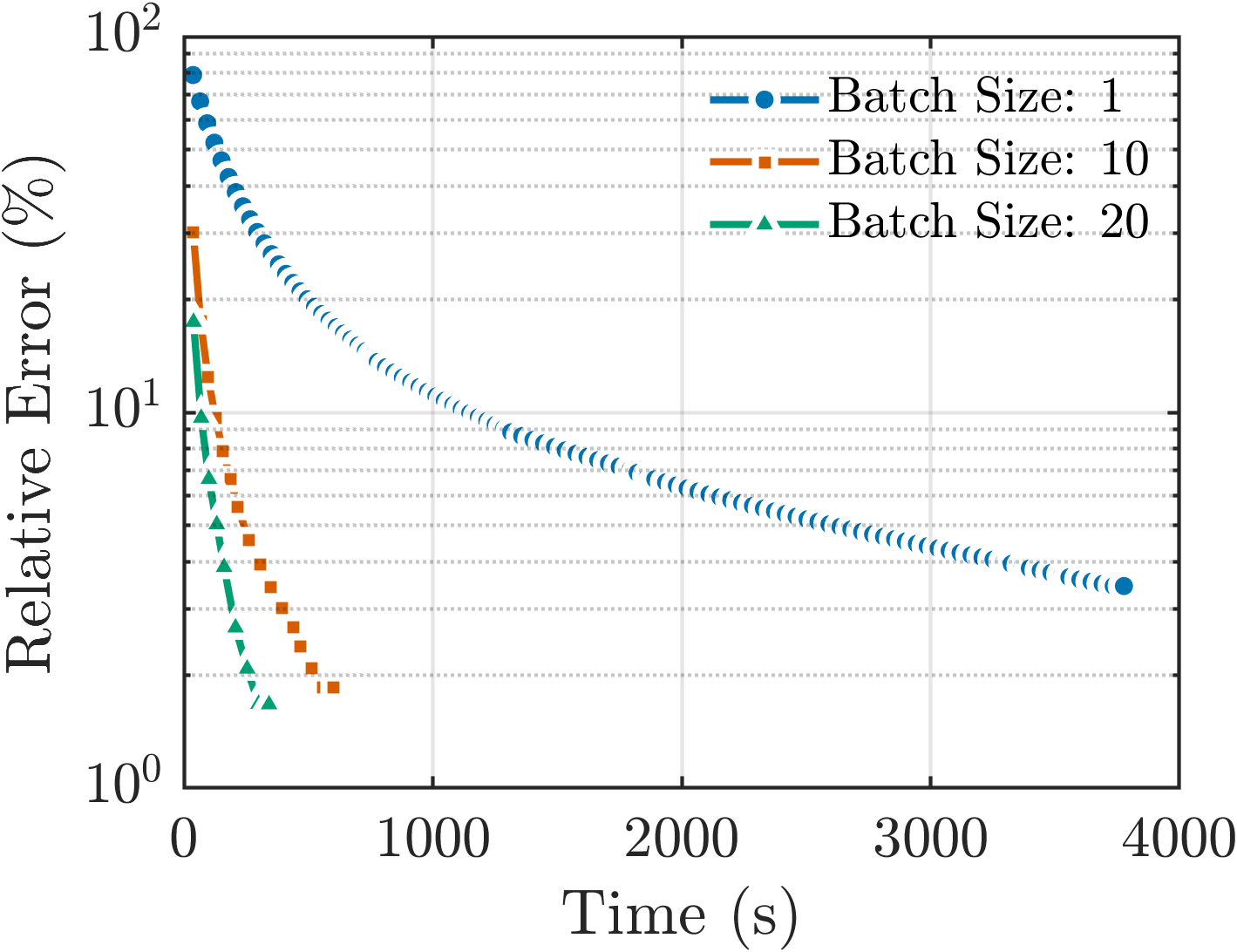}
			\caption{Sensitivity of R3-Trust to the batch size (the rank of the
				low-rank term added per iteration) on the prostate case, for batch
				sizes $1$, $10$, and $20$. Increasing the batch size reduces runtime
				(right) but can slightly degrade the error--nnz trade-off (left). A
				batch size of $10$ is used as the default.}
			\label{fig:r3trust_parameters_sensitivity_analysis_prostate}
		\end{figure}
		
		\section*{Acknowledgments}
		
		The authors from Memorial Sloan Kettering Cancer Center gratefully acknowledge partial support from the NIH Cancer Center Support Grant/Core Grant (P30 CA008748).
		
		\bibliographystyle{siamplain}
		\bibliography{references}

\begin{thebibliography}{10}

\bibitem{Aybat2015-fo}
{\sc N.~S. Aybat and G.~Iyengar}, {\em An alternating direction method with
  increasing penalty for stable principal component pursuit}, Comput. Optim.
  Appl., 61 (2015), pp.~635--668.

\bibitem{Babacan2012-gm}
{\sc S.~D. Babacan, M.~Luessi, R.~Molina, and A.~K. Katsaggelos}, {\em Sparse
  bayesian methods for {Low-Rank} matrix estimation}, IEEE Trans. Signal
  Process., 60 (2012), pp.~3964--3977.

\bibitem{beck2009fast}
{\sc A.~Beck and M.~Teboulle}, {\em A fast iterative shrinkage-thresholding
  algorithm for linear inverse problems}, SIAM journal on imaging sciences, 2
  (2009), pp.~183--202.

\bibitem{boyd2011distributed}
{\sc S.~Boyd, N.~Parikh, E.~Chu, B.~Peleato, J.~Eckstein, et~al.}, {\em
  Distributed optimization and statistical learning via the alternating
  direction method of multipliers}, Foundations and Trends{\textregistered} in
  Machine learning, 3 (2011), pp.~1--122.

\bibitem{cabral2013unifying}
{\sc R.~Cabral, F.~De~la Torre, J.~P. Costeira, and A.~Bernardino}, {\em
  Unifying nuclear norm and bilinear factorization approaches for low-rank
  matrix decomposition}, in Proceedings of the IEEE International Conference on
  Computer Vision (ICCV), 2013, pp.~2488--2495.

\bibitem{cai2010singular}
{\sc J.-F. Cai, E.~J. Cand{\`e}s, and Z.~Shen}, {\em A singular value
  thresholding algorithm for matrix completion}, SIAM Journal on optimization,
  20 (2010), pp.~1956--1982.

\bibitem{Candes2009-fm}
{\sc E.~J. Cand{\`e}s, X.~Li, Y.~Ma, and J.~Wright}, {\em Robust principal
  component analysis?}, Journal of the ACM, 58 (2011), pp.~11:1--11:37.

\bibitem{Chandrasekaran2011-vc}
{\sc V.~Chandrasekaran, S.~Sanghavi, P.~A. Parrilo, and A.~S. Willsky}, {\em
  {Rank-Sparsity} incoherence for matrix decomposition}, SIAM J. Optim., 21
  (2011), pp.~572--596.

\bibitem{chen2021scatterbrain}
{\sc B.~Chen, T.~Dao, E.~Winsor, Z.~Song, A.~Rudra, and C.~R{\'e}}, {\em
  Scatterbrain: Unifying sparse and low-rank attention approximation}, in
  Advances in Neural Information Processing Systems (NeurIPS), vol.~34, 2021,
  pp.~17413--17426.

\bibitem{conn2000trust}
{\sc A.~R. Conn, N.~I.~M. Gould, and P.~L. Toint}, {\em Trust Region Methods},
  MPS-SIAM Series on Optimization, SIAM, Philadelphia, PA, 2000.

\bibitem{ding2011bayesian}
{\sc X.~Ding, L.~He, and L.~Carin}, {\em Bayesian robust principal component
  analysis}, IEEE Transactions on Image Processing, 20 (2011), pp.~3419--3430.

\bibitem{eckart1936approximation}
{\sc C.~Eckart and G.~Young}, {\em The approximation of one matrix by another
  of lower rank}, Psychometrika, 1 (1936), pp.~211--218.

\bibitem{fazel2002matrix}
{\sc M.~Fazel}, {\em Matrix rank minimization with applications}, PhD thesis,
  PhD thesis, Stanford University, 2002.

\bibitem{golub2013matrix}
{\sc G.~H. Golub and C.~F. Van~Loan}, {\em Matrix Computations}, Johns Hopkins
  University Press, Baltimore, MD, 4th~ed., 2013.

\bibitem{doi:10.1137/090771806}
{\sc N.~Halko, P.~G. Martinsson, and J.~A. Tropp}, {\em Finding structure with
  randomness: Probabilistic algorithms for constructing approximate matrix
  decompositions}, SIAM Review, 53 (2011), pp.~217--288,
  \url{https://doi.org/10.1137/090771806}.

\bibitem{hogben2013handbook}
{\sc L.~Hogben}, {\em Handbook of linear algebra}, CRC press, 2013.

\bibitem{jhanwar2023portpy}
{\sc G.~Jhanwar, M.~Tefagh, V.~T. Taasti, S.~R. Alam, S.~Tuomaala, S.~Nadeem,
  and M.~Zarepisheh}, {\em {PortPy}: An open-source {Python} package for
  planning and optimization in radiation therapy including benchmark data and
  algorithms}, AAPM 65th Annual Meeting \& Exhibition,  (2023).

\bibitem{li2023losparse}
{\sc Y.~Li, Y.~Yu, Q.~Zhang, C.~Liang, P.~He, W.~Chen, and T.~Zhao}, {\em
  Losparse: Structured compression of large language models based on low-rank
  and sparse approximation}, in International Conference on Machine Learning
  (ICML), PMLR, 2023, pp.~20336--20350.

\bibitem{liu2010interior}
{\sc Z.~Liu and L.~Vandenberghe}, {\em Interior-point method for nuclear norm
  approximation with application to system identification}, SIAM Journal on
  Matrix Analysis and Applications, 31 (2010), pp.~1235--1256.

\bibitem{10.1093/qmath/11.1.50}
{\sc L.~Mirsky}, {\em {Symmetric gauge functions and unitarily invariant
  norms}}, The Quarterly Journal of Mathematics, 11 (1960), pp.~50--59,
  \url{https://doi.org/10.1093/qmath/11.1.50}.

\bibitem{nesterov2007gradient}
{\sc Y.~Nesterov}, {\em Gradient methods for minimizing composite functions},
  Mathematical Programming, 140 (2013), pp.~125--161.

\bibitem{Parikh2014-ft}
{\sc N.~Parikh and S.~Boyd}, {\em Proximal algorithms}, Foundations and
  Trends\textregistered{} in Optimization, 1 (2014), pp.~127--239.

\bibitem{romeijn2005column}
{\sc H.~E. Romeijn, R.~K. Ahuja, J.~F. Dempsey, and A.~Kumar}, {\em A column
  generation approach to radiation therapy treatment planning using aperture
  modulation}, SIAM Journal on Optimization, 15 (2005), pp.~838--862.

\bibitem{Shen2014-tq}
{\sc Y.~Shen, Z.~Wen, and Y.~Zhang}, {\em Augmented lagrangian alternating
  direction method for matrix separation based on low-rank factorization},
  Optim. Methods Softw., 29 (2014), pp.~239--263.

\bibitem{shepard1999optimizing}
{\sc D.~M. Shepard, M.~C. Ferris, G.~H. Olivera, and T.~R. Mackie}, {\em
  Optimizing the delivery of radiation therapy to cancer patients}, SIAM
  Review, 41 (1999), pp.~721--744.

\bibitem{tefagh2025compressed}
{\sc M.~Tefagh, G.~Jhanwar, and M.~Zarepisheh}, {\em Compressed radiotherapy
  treatment planning: A new paradigm for rapid and high-quality treatment
  planning optimization}, Medical Physics,  (2025),
  \url{https://doi.org/10.1002/mp.17736}.

\bibitem{zarepisheh2018computation}
{\sc M.~Zarepisheh, L.~Xing, and Y.~Ye}, {\em A computation study on an
  integrated alternating direction method of multipliers for large scale
  optimization}, Optimization Letters, 12 (2018), pp.~3--15.

\bibitem{zhao2014robust}
{\sc Q.~Zhao, D.~Meng, Z.~Xu, W.~Zuo, and L.~Zhang}, {\em Robust principal
  component analysis with complex noise}, in International Conference on
  Machine Learning (ICML), PMLR, 2014, pp.~55--63.

\bibitem{Zhou_undated-ur}
{\sc T.~Zhou and D.~Tao}, {\em {GoDec}: Randomized low-rank \& sparse matrix
  decomposition in noisy case}, in Proceedings of the 28th International
  Conference on Machine Learning ({ICML}), 2011, pp.~33--40.

\bibitem{Zhou2013-tj}
{\sc T.~Zhou and D.~Tao}, {\em Greedy bilateral sketch, completion \&
  smoothing}, in Proceedings of the Sixteenth International Conference on
  Artificial Intelligence and Statistics, C.~M. Carvalho and P.~Ravikumar,
  eds., vol.~31 of Proceedings of Machine Learning Research, Scottsdale,
  Arizona, USA, 2013, PMLR, pp.~650--658.

\bibitem{Zhou2014-ug}
{\sc X.~Zhou, C.~Yang, H.~Zhao, and W.~Yu}, {\em {Low-Rank} modeling and its
  applications in image analysis}, ACM Comput. Surv., 47 (2014), pp.~1--33.

\bibitem{Zhou2010-ji}
{\sc Z.~Zhou, X.~Li, J.~Wright, E.~Cand{\`e}s, and Y.~Ma}, {\em Stable
  principal component pursuit}, in 2010 {IEEE} International Symposium on
  Information Theory, June 2010, pp.~1518--1522.

\end{thebibliography}
		
	\end{document}